\documentclass[a4paper,11pt]{amsart}
\usepackage[utf8]{inputenc}
\usepackage[T1]{fontenc}
\usepackage[english]{babel}
\usepackage{mathtools,amssymb,amsthm}
\usepackage{enumitem}
\usepackage{esint}
\usepackage{lmodern}
\usepackage{geometry}
\usepackage{amsbsy}
\usepackage{bm}
\usepackage{mathrsfs}
\usepackage{physics}
\usepackage{float}
\usepackage{xcolor}
\usepackage{hyperref}
\usepackage{algorithm}
\usepackage{algpseudocode}
\usepackage[section]{placeins} % for the figure placing

\usepackage{ulem}
\usepackage{color}

\renewcommand \L{\mathrm{L}}

\newcommand \Om{\Omega}

\newcommand \eps{\epsilon}
\newcommand \1{\mathrm{1}}

\newcommand \Z{\mathbb{Z}}
\newcommand \N{\mathbb{N}}

\newcommand \B{\operatorname{B}}
\newcommand \UB{\operatorname{UB}}
\newcommand \R{\mathbb{R}}

\renewcommand \S{\mathbb{S}}

\renewcommand{\P}{\operatorname{P}}

\newcommand{\sm}{\setminus}
\newcommand{\lra}{\longrightarrow}
\newcommand{\rau}{\rightharpoonup}

\newcommand{\nif}{{n \rightarrow +\infty}}

\renewcommand{\div}{\operatorname{div}}
\renewcommand{\H}{\operatorname{H}}
\newcommand{\Span}{\operatorname{span}}

\newtheorem{theorem}{Theorem}

\newtheorem{remark} [theorem]{Remark}
\newtheorem{proposition} [theorem]{Proposition}

\newtheorem{conjecture}{Conjecture}
\theoremstyle{definition}

\usepackage[colored]{shadethm}

\begin{document}

%%%%%%%%%%%%%%%%%%%%%%%%%%%%%%%%%%%%%%%%%%%%%%%%%%%%%%%%%%
 %amsart format
\title[]{Numerical optimization of Neumann eigenvalues of domains in the sphere}

\author[E. Martinet]
{Eloi Martinet}
\address[Eloi Martinet]{Univ. Savoie Mont Blanc, CNRS, LAMA \\
73000 Chamb\'ery, France}
\email[E. Martinet]{eloi.martinet@univ-smb.fr}

\subjclass[2010]{}

\date{\today}

\begin{abstract}
This paper deals with the numerical optimization of the first three eigenvalues of the Laplace-Beltrami operator of domain in the Euclidean sphere in $\R^3$ with Neumann boundary conditions. We address two approaches : the first one is a generalization of the initial problem leading to a density method and the other one is a shape optimization procedure via the level-set method. The original goal of those method was to investigate the conjecture according to which the geodesic ball were optimal for the first non-trivial eigenvalue under certain conditions. These computations give some strong insight on the optimal shapes of those eigenvalue problems and show a rich variety of shapes regarding the proportion of the surface area of the sphere occupied by the domain. In a last part, the same algorithms are used to carry the same survey on a torus.
\end{abstract}

\maketitle

\section{Introduction}

Let $n\geq 1$. The problem we consider in the first part of this paper involves the spectrum of the Laplace operator on domain in the unit sphere $\S^n = \{ x \in \R^{n+1} : \|x\| = 1 \}$. Let $\Omega \subset \S^n$ be an open, bounded and Lipschitz set. The spectral theorem assures that the problem
$$
\begin{cases}
-\Delta u = \mu_k(\Om) u \mbox { in } \Om,\\
\frac{\partial u}{\partial \nu} = 0  \mbox { on } \partial \Om,
\end{cases}
$$
with $u \in H^1(\Om)\sm \{0\}$ has a sequence of eigenvalues
$$
0 = \mu_0(\Om) \leq \mu_1(\Om) \leq \mu_2(\Om) \leq ... \to +\infty.
$$
The corresponding eigenfunctions satisfy the well-known variational formula
\begin{equation}\label{domain_var}
    \mu_k(\Om) = \min_{V\in{\mathcal V}_{k+1}} \max_{u \in V\sm \{0\}} \frac{\int_\Om |\nabla u|^2}{\int_\Om u^2},
\end{equation}
where ${\mathcal V}_k$ is the family of subspaces of dimension $k$ in $H^1(\Om)$. We are interested in the following problem : for $m > 0$ and $k \in \N$, solve
\begin{equation}\label{domain_problem}
\sup \{ \mu_k(\Om) : \Om \subset \S^n, | \Om | = m, \Om \mbox{ bounded, open and Lipschitz} \}.
\end{equation}
The aim of this paper will be to study this optimization problem from a numerical point of view.

While the numerical shape optimization of Neumann eigenvalues of domains in the Euclidian space have drawn a lot of attention in the past years (see for instance \cite{antunes2012numerical} \cite{antunes_numerical_2017} \cite{antunes_oudet} \cite{BMO_2022}), the litterature on the optimization of those eigenvalues for domains in curved spaces is sparse. The present work will address this problem by considering the optimization of several Neumann eigenvalues of domains in the sphere $\S^n$ as well as on a torus.

Old and recent theoretical works have been conducted to prove the optimality of the spherical cap for domain in the sphere for the first Neumann eigenvalue, either by assuming simple connectedness of the optimal domain \cite{bandle_isoperimetric_1972} \cite{langford2022maximizers} or, following the seminal work of Weinberger \cite{weinberger_isoperimetric_1956}, using the so-called "mass transplantation technique". However all these results require some restriction on the domain : for instance, to lie on the hemisphere \cite{ashbaugh_sharp_1995} or to lie outside of some spherical cap \cite{bucur_sharp_2022}.

In this work, we will provide some numerical evidence of properties of the optimal domain in the sphere. Especially, we will see that some density-based method strongly suggest that the problem of the optimization of the first Neumann eigenvalue on domains can not be tackled by mass transplantation arguments for domains and densities of masses large enough. Moreover, we will see how the numerical study of the second non-trivial Neumann eigenvalue suggests that the optimal shape is two disjoint geodesic balls \cite{bucur_sharp_2022}. We will also numerically witness the rich variety of optimal shapes for the third non-trivial eigenvalue. In a last part we will consider the same questions on a torus.

As it has already been said, one of the numerical method will rely on some relaxation of the original problem by extending the relation \eqref{domain_var} to the class of densities $\rho \in \L^\infty(\S^n, [0,1])$ in the following way
\begin{equation}\label{density_var}
 \mu _k(\rho) := \inf_{V\in{\mathcal V}_{k+1}} \max_{u \in V \sm \{0\}} \frac{\int_{\S^n} \rho|\nabla u|^2}{\int_{\S^n} \rho u^2},\end{equation}
where ${\mathcal V}_{k+1}$ is the family of subspaces of dimension $k+1$ in
\begin{equation}
\{u\cdot 1_{\{\rho (x)>0\}}: u \in C^\infty_c (\S^n)\}.
\end{equation}

This relaxation, which has already been extensively used in \cite{BMO_2022}, can also be found in greater generality in \cite{colbois_spectrum_2019}.

The original problem will then be replaced by

\begin{equation} \label{density_problem}
\sup \left\{ \mu_k(\rho) :  \rho : \S^n \rightarrow [0,1], \int_{\S^n}\rho=m\right\}.
\end{equation}
which well-posedness is established below.

This formulation allows to perform some classical optimization methods such as gradient descent over the variable $\rho$ instead of considering difficult shape optimization problem involving changes of topology. However, the density method is strictly more general than the problem \eqref{domain_problem} in the sense that optimal density may not correspond to characteristic functions of domains. This led to an implementation of a shape optimization method, namely the level set method, which solves directly the problem \eqref{domain_problem} with possible changes in topology.

In the next sections, we first see the theoretical aspects brought by the relaxation \eqref{density_var}. Then we discuss the practical implementation of the two methods cited above and provide numerical results related to the first three eigenvalues. We also address some possible consequences these computations have at a theoretical level.

\section{Existence and approximation of the optimal density}

We start by setting a theoretical framework which makes problem \eqref{density_problem} well-posed. The results follows using the same ideas as in \cite{BMO_2022}. For this reason, we shall not enter too much in the details.

A natural question that directly arises is the one of the existence of an optimal density. In the Euclidian case $\R^n$, we had to rely on some concentration-compactness result to obtain the existence of some \textit{collection} of densities. Here the fact that the sphere has finite measure allows us to give a complete existence result easily :

\begin{theorem}[Existence of an optimal density] \label{th:existence}
  Let $0 \leq m \leq |\S^n|$. Then problem (\ref{density_problem}) has a solution.
\end{theorem}

\begin{proof} We actually prove the upper-semicontinuity of the eigenvalues with respect to the weak-* convergence. Let $k \in \N$ and $\rho, (\rho_n)_{n \in \N} \in \L^1(\S^n, [0,1])$ be functions such that $\rho_n \rau \rho$. Let $\eps > 0$ and $V = \Span\{v_0 \1_{\{\rho>0\}},...,v_k \1_{\{\rho>0\}} \}$ be such that
\[
\mu_k(\rho) \geq \max_{u\in V\setminus \{0\}} \frac{\int_{\S^n} \rho|\nabla u|^2 }{\int_{\S^n} \rho u^2} - \eps.
\]

Let us consider $V_n = < v_0 \1_{\{\rho_n>0\}},...,v_k \1_{\{\rho_n>0\}}>$ and $u_n = \sum_{i=0}^k \alpha_i^n v_i$ a maximizing sequence in
\[
\max_{u\in V_n\setminus \{0\}} \frac{\int_{\S^n} \rho_n |\nabla u|^2}{\int_{\S^n} \rho_n u^2}.
\]

For $n$ large enough, $V_n$ is of dimension $k+1$ hence

\[
\mu_k(\rho_n) \leq \frac{\int_{\S^n} \rho_n|\nabla u_n|^2}{\int_{\S^n} \rho_n u_n^2}.
\]

Note that we can suppose by homogeneity that $\sum_{i=0}^k (\alpha_i^n)^2=1$ for all $n$. Up to a subsequence, we get that $\alpha_i^n \rightarrow \alpha_i \in \R$ for all $i$. By putting $\tilde v = \sum_{i=0}^k \alpha_i v_i$ we get that
\[
\int_{\S^n} \rho_n |\nabla u_n|^2 \rightarrow \int_{\S^n} \rho |\nabla \tilde v|^2
\]
and
\[
\int_{\S^n} \rho_n u_n^2 \rightarrow \int_{\S^n} \rho \tilde v^2.
\]
Thus
\[ \mu_k(\rho) \geq \limsup_{\nif} \mu_k(\rho_n) - \eps.\]
This relation being valid for all $\eps$, we get
\[ \mu_k(\rho) \geq \limsup_{\nif} \mu_k(\rho_n).\]

Now let $(\rho_n)_{n\in\N}$ be some maximizing sequence of the problem (\ref{density_problem}). There exists a $\rho \in \L^1(\S^n, [0,1])$ such that up to a subsequence, $\rho_n \rau \rho$ weakly. By upper-semicontinuity, we get that $\mu_k(\rho) \geq \limsup_{\nif} \mu_k(\rho_n)$ and the fact that $1 \in \L^1(\S^n, \R)$ ensure that the condition $\int_{\S^n} \rho_n = m$ is satisfied at the limit.
\end{proof}

From a numerical point of view, computing the generalized eigenvalue \textit{via} finite element method is not possible in general due to the potential vanishing of $\rho$ on some non-negligible parts of $\S^n$. It is possible to approximate our generalized eigenvalues by well-defined ones of non-zero densities :

\begin{theorem}[Approximation]
  \label{th:approximation}
  Let $\rho \in \L^1(\S^n,[0,1])$ , $\int_{\S^n} \rho = m > 0$. We introduce the following quantity :
  \[
  \mu_k^\eps(\rho) := \min_{V\in{\mathcal V}_{k+1}} \max_{u \in V \sm \{0\}} \frac{\int_{\S^n} (\rho+ \eps) |\nabla u|^2}{\int_{\S^n} (\rho+\eps^2) u^2}
  \]
  where ${\mathcal V}_{k+1}$ is the family of subspace of dimension $k+1$ in $\H^1(\S^n)$.

  Then :
  \[
  \mu_k^\eps(\rho) \xrightarrow[\eps \to 0]{} \mu_k(\rho).
  \]
\end{theorem}

For the proof in the Euclidian case, we refer to \cite[Lemma 14]{BMO_2022}.

\begin{remark}
$\mu_k^\eps(\rho)$ is the $k$-th non-trivial eigenvalue of the well posed elliptic problem
$$
 -\div[(\rho+\eps)\nabla u] = \mu_k^\eps(\rho) (\rho+\eps^2) u
$$
on $\S^n$.
\end{remark}

\begin{proof}
The proof decomposes into proving both the limsup and the liminf. The limsup is proven in the same way as in previous theorem; let us focus on the liminf.

Let $u_0^\eps,...,u_k^\eps \in \H^1(\S^n)$ be the eigenfunctions associated to the eigenvalues $\mu_0^\eps,...,\mu_k^\eps$, orthogonal and normalized in the sense that
\begin{equation}
\label{eqn:orthogonality}
\int_{\S^n} (\rho+\eps^2)u_i^\eps u_j^\eps = \delta_{i,j}.
\end{equation}
This implies that
\begin{equation}
\label{eqn:bounded_grad}\int_{\S^n} (\rho+\eps) |\nabla u_i^\eps|^2 = \mu_i^\eps(\rho)
\end{equation}
and this quantity can be considered bounded independently of $i$ and $\eps$ by some bound M. If not, the limsup would be infinite and the previous case would allow us to conclude. From equation \ref{eqn:orthogonality} we deduce that $(\eps u_i^\eps)_\eps$ is bounded in $\L^2(\S^n)$ for all $i$. Hence we can find a subsequence such that for all $0 \leq i \leq k$, the sequence $(\eps u_i^\eps)$ converges weakly in $\L^2(\S^n)$ to some function $g_i$. Denoting $\bar{v} = \frac{1}{|\S^n|}\int_{\S^n} v$, we get  $\eps \bar{u_i^\eps} \lra \bar{g_i}$ for all $i$, the constant function $1$ being in $\L^2$. Moreover, the sequence $(\sqrt{\eps}\nabla u_i^\eps)_\eps$ is bounded in $\L^2(\S^n)$ hence by  we get by the Poincaré-Wirtinger inequality :
\[
\| \eps u_i^\eps - \eps \bar{u_i^\eps} \|_{\L^2} \leq C \|\eps \nabla u_i^\eps \|_{\L^2} \lra 0.
\]
We deduce that $\eps u_i^\eps \lra \bar{g_i}$ strongly in $\L^2$. We can then conclude that $\bar{g_i} = 0$ by noticing that
\[
0 = \lim_{\eps\to 0} \eps^2 \int_{\S^n} \rho (u_i^\eps)^2 = \int_{\S^n} \rho \bar{g_i}^2 = \bar{g_i}^2 m.
\]
By Cauchy-Schwarz inequality, this implies that $\int_{\S^n} u_i^\eps u_j^\eps \lra 0$ which in turn results in $\int_{\S^n} \eps^2 (v^\eps)^2 \dd x \lra 0$ for all $v^\eps \in \Span\{u_0^\eps, ..., u_k^\eps\}$. Using this last limit and the fact that $\Span\{u_0^\eps 1_{\{\rho>0\}}, ...,u_k^\eps 1_{\{\rho>0\}} \}$ is of dimension $k+1$ for $\eps$ small enough we finally get
\[
\mu_k(\rho)
= \inf_{V\in{\mathcal V}^\rho_{k+1}} \max_{u \in V \sm \{0\}} \frac{\int_{\S^n} \rho|\nabla u|^2}{\int_{\S^n} \rho u^2}
\leq \liminf_{\eps \to 0} \max_{v \in \Span\{u_0^\eps, ..., u_k^\eps \}}  \frac{\int_{\S^n} (\rho+ \eps) |\nabla u|^2}{\int_{\S^n} (\rho+\eps^2) u^2}
= \liminf_{\eps \to 0} \mu_k^\eps(\rho)
\]
which concludes the proof.
\end{proof}

\begin{theorem}[Approximation of maxima]Let $0 < m \leq  |\S^n|$. Then
\begin{equation}
    \max\limits_{\|\rho\|_{\L^1}=m} \mu_k^\eps(\rho) = \max\limits_{\|\rho\|_{\L^1}=m} \mu_k(\rho).
\end{equation}
\end{theorem}

\begin{proof}Let $(\rho^\eps)_\eps$ be such that $\mu_k^\eps(\rho^\eps) = \max\limits_{\|\rho\|_{\L^1}=m} \mu_k^\eps(\rho) $ and $\rho^*$ be such that $\mu_k(\rho^*) = \max\limits_{\|\rho\|_{\L^1}=m} \mu_k(\rho)$. In the same way than the upper semicontinuity in theorem ~\ref{th:existence}, we have that \[
\limsup_{\eps \lra 0} \mu_k^\eps(\rho^\eps) \leq \mu_k(\tilde \rho) < \max_{\|\rho\|_{\L^1}=m} \mu_k(\rho) < +\infty
\]
whenever $\rho^\eps \rightharpoonup \tilde \rho$ weakly in $\L^\infty$. Hence the sequence $(\mu_k^\eps(\rho^\eps))_\eps$ is bounded. Then the previous lower-semicontinuity result implies that
\[
\max\limits_{\|\rho\|_{\L^1}=m} \mu_k(\rho)
= \mu_k(\rho^*)
\leq \liminf_{\eps \to 0} \mu_k^\eps(\rho^\eps)
=  \liminf_{\eps \to 0} \max\limits_{\|\rho\|_{\L^1}=m} \mu_k^\eps(\rho^\eps)
\]
which concludes the proof.
\end{proof}

\section{Density method}

In this section we discuss the numerical implementation of the density method, which follows the same lines as \cite{BMO_2022} with some new technical difficulties working on the sphere. The sphere is assumed to be discretized by a mesh that remains the same during the  optimization process. Let $V_h= \Span(\phi_1,...,\phi_n) \subset \H^1(\S^N)$ be a finite element space. If $v = \sum_i v_i \phi_i$ we denote by $\bar v = (v_1,...,v_n)^T$ its coordinates on the basis $(\phi_i)_i$. Even if we made the choice here to discretize both the density and the eigenfunctions on the same FE space $V_h$, we could have considered different ones as we did in our previously cited work. Let $\rho \in V_h$ be a density (i.e. $\rho : \S^n \to [0,1]$). We denote by $\bar \mu_k^\eps(\rho)$ the eigenvalue of the finite-dimensional eigenvalue problem :
\begin{equation}
    \label{eqn:fe_eigenvalue}
    \bm{M}^\eps(\rho) \bar u_k^\eps(\rho) = \bar \mu_k^\eps(\rho) \bm{K}^\eps(\rho) \bar u_k^\eps(\rho)
\end{equation}

where
\[
    \bm{M}^\eps(\rho) = \left( \int_{\S^n} (\rho+\eps) \nabla \phi_i \nabla \phi_j \right)_{i,j}
\]
and
\[
    \bm{K}^\eps(\rho) = \left( \int_{\S^n} (\rho+\eps^2) \phi_i \phi_j \right)_{i,j}.
\]
Since we use a gradient-based optimization method, we need to differentiate $\bar \mu_k^\eps(\rho)$ with respect to its coordinates $(\rho_1,...,\rho_n)$ in the basis $(\phi_i)_i$. Assuming that $ \bar \mu_k^\eps(\rho)$ is simple, we can differentiate the equation \eqref{eqn:fe_eigenvalue} and multiply on the left by $(\bar u_k^\eps(\rho))^T$ to get
\begin{equation}
    \partial_l \bar \mu_k^\eps = \frac{(\bar u_k^\eps)^T\left( \partial_l \bm{M}^\eps  - \bar \mu_k^\eps \partial_l \bm{K}^\eps \right) \bar u_k^\eps}{(\bar u_k^\eps)^T  \bm{K}^\eps(\rho) \bar u_k^\eps}
\end{equation}
where the derivatives of the matrices are respectively
\[
    \partial_l \bm{M}^\eps(\rho) = \left( \int_{\S^n} \phi_l \nabla \phi_i \nabla \phi_j \right)_{i,j}
\]
and
\[
    \partial_l \bm{K}^\eps(\rho) = \left( \int_{\S^n} \phi_l \phi_i \phi_j \right)_{i,j}.
\]

One may remember that since we are on the sphere, we don't have scale-homogeneity of the eigenvalue as it is the case in $\R^n$. Hence we have to enforce the condition $\int_{\S^n} \rho = m$, which in the discrete case becomes
$\bar \rho \cdot \bm g = m$ where
\[
  \bm g = \left(\int_{\S^n} \phi_l\right)_l.
\]

\subsection{Multiple eigenvalues.}

One of the main hurdles in spectral shape optimization is to handle the multiplicity of the eigenvalues. Indeed, in practice, the optimal density is expected to have high multiplicity. The method presented in \cite{BMO_2022} consisted in adding a constraint forcing the eigenvalues that were close (depending on a certain threshold $\sigma$) to get them even closer. While it yielded good results in the planar case, it suffered one peculiar issue here, especially in the case of $\mu_1$. Indeed, even if the optimal density seems to be of multiplicity $2$, the third eigenvalue is observed to be close to $\mu_2$ leading the previous method either to fall into a local maximum of multiplicity $3$ if $\sigma$ was too high, or to be too unstable to converge if $\sigma$ was too small.
One way to overcome this issue is to compute a better direction than the one given by the gradient of $\mu_k$. Several numerical methods have been investigated in this direction, such as finding the direction $h$ as a solution to the problem
\[
    \max_{h \in V_h, \| h \| = 1} \min \left\{ \dd \mu_k^\eps(\rho). h, ..., \dd \mu_{k+m-1}^\eps(\rho).h \right\}
\]
where $m$ is the "guessed" multiplicity, chosen such that $\mu_{k+m-1}^\eps(\rho) - \mu_k^\eps(\rho) \leq \sigma$ and $\mu_{k+m}^\eps(\rho) - \mu_k^\eps(\rho) > \sigma$. Such method has been used for instance in\cite{antunes_numerical_2017}. Heuristically, it produces a direction that increases all the selected eigenvalues by a maximal amount. Another, more rigorous method would be to consider the true directional derivative of our multiple eigenvalue in the direction $h$ (which always exists), namely
\[
    (\mu^\eps_k)'(\rho)(h) := \lim_{t \to 0^+} \frac{\mu^\eps_k(\rho+th) - \mu_k^\eps(\rho)}{t}
\]
and in the same way as before, to search the direction that maximizes this variation :
\[
    \max_{h \in V_h, \| h \| = 1} (\mu^\eps_k)'(\rho)(h).
\]
This has been brilliantly studied in \cite{de_gournay_velocity_2006} where the author shows that the search of the direction can be efficiently solved by semi-definite programming methods.

However, in our case, these methods did not seem to perform better that our previous one. It turned out that a simple modification of our problem leads to a very good convergence. Instead of searching for a direction based on directional derivatives, we regularize our problem to make it differentiable. It can then be handled better from the interior point optimizer. The idea is the following : suppose that $\mu_k^\eps(\rho)$ stays away from $\mu_{k-1}^\eps(\rho)$ during the whole optimization process (which is always the case in practice) and is part of a cluster $\mu_{k}^\eps(\rho), ..., \mu_{k+m-1}^\eps(\rho)$ of eigenvalues closer than $\sigma$. Then obviously

\begin{equation}
    \mu_{k}^\eps(\rho) = \min \left\{\mu_{k}^\eps(\rho), ..., \mu_{k+m-1}^\eps(\rho)\right\}.
\end{equation}

But for $x_0,...,x_{m-1} > 0$, we have that
\begin{equation}
    \min \left\{ x_0,..., x_{m-1}\right\} = \lim_{p\to +\infty} \left( \sum_{i} x_i^{-p}\right)^{-1/p}
\end{equation}

so that by taking $p$ large, we can approximately write
\begin{equation}
    \min \left\{ x_0,..., x_{m-1}\right\} \approx \left( \sum_{i} x_i^{-p}\right)^{-1/p}.
\end{equation}

In this spirit, we instead optimize the functional
\begin{equation}
    \rho \mapsto \left( \sum_{i} \mu_{k+i}^\eps(\rho)^{-p}\right)^{-1/p}.
\end{equation}

This function being a symmetric function of the eigenvalues, it is expected to be smooth near $\rho$ if $m$ is the actual multiplicity of $\mu_k^\eps(\rho)$ \cite{kato2013perturbation}. If it is not, the largest eigenvalues of the cluster are mostly ignored.

\subsection{Dimension 1}

After running a series of simulations on the sphere, it appeared that the optimal density seems to be axially symmetric in the case of $\mu_1$. In order to get even more insight on the density problem, we subsequently ran simulations only in 1D, considering the density $\rho$ as a real function of the latitude $\theta \in [0,\pi]$, i.e. the angle from one pole to a point on the sphere. By separation of variables, if $\rho : \S^2 \to [0,1]$ is axially symmetric and not degenerated, then $\mu_1(\rho)$ is the least non-zero eigenvalue of the following two eigenvalue problems

\begin{equation*}
    \begin{cases}
        -\frac{1}{\sin(\theta)} \frac{\dd}{\dd \theta} \left( \rho(\theta) \sin(\theta) \frac{\dd y}{\dd \theta} \right) + \frac{\rho(\theta)}{\sin^2(\theta)} y = \rho(\theta) \bar \mu y \mbox{ on } (0, \pi) \\
        y(0) \mbox{ and } y(\pi) \mbox{ are finite}
    \end{cases}
\end{equation*}

\begin{equation*}
    \begin{cases}
        -\frac{1}{\sin(\theta)} \frac{\dd}{\dd \theta} \left( \rho(\theta) \sin(\theta) \frac{\dd y}{\dd \theta} \right) = \rho(\theta) \tilde \mu y \mbox{ on } (0, \pi) \\
        y(0) \mbox{ and } y(\pi) \mbox{ are finite}
    \end{cases}.
\end{equation*}
See \cite{ashbaugh_sharp_1995} for more details. Formally, by developping the derivative in the first differential equation, we can see that the condition $y(0)=y\pi)=0$ is forced by the term in $\frac{1}{\sin(\theta)}$, which penalizes large values of $y$ in $0$ and $\pi$. In the same way, we can see that in the second differential equation, the boundary conditions needs to be $y'(0)=y'(\pi)=0$.

As before, since we allow $\rho$ to vanish, we need to regularize the problem to make it elliptic. The problems that are actually solved are

\begin{equation*}
    \begin{cases}
        -\frac{\dd}{\dd \theta} \left( \left(\rho(\theta) \sin(\theta) + \eps \right)\frac{\dd y}{\dd \theta} \right) + \frac{\rho(\theta)+\eps}{\sin(\theta)} y = (\rho(\theta)\sin(\theta)+\eps^2) \bar \mu y \mbox{ on } (0, \pi) \\
        y(0)=y(\pi) =0
    \end{cases}
\end{equation*}

\begin{equation*}
    \begin{cases}
        -\frac{\dd}{\dd \theta} \left( \left(\rho(\theta) \sin(\theta) + \eps \right) \frac{\dd y}{\dd \theta} \right) = (\rho(\theta)\sin(\theta)+\eps^2) \tilde \mu y \mbox{ on } (0, \pi) \\
        y'(0)=y'(\pi)=0
    \end{cases}
\end{equation*}
where $\eps$ is supposed to be small.

\subsection{Numerical considerations.}

Our optimization procedure is carried out by IPOPT \cite{wachter2006implementation} while the finite element computations is perfomed in GetFEM \cite{MR4199501}. A first optimization is carried on a coarse mesh of $2246$ vertices with $\P_1$ finite elements. A second optimization is then performed with the result of the previous one as initilization on a mesh consisting in $35401$ elements (the meshes that are used can be vizualized Figure ~\ref{fig:meshes_density}). For each $m$, the optimization is performed multiple times with different initialization and the density giving
the best value is finally kept.

\begin{figure}
    \centering
    \includegraphics[width=0.33\textwidth]{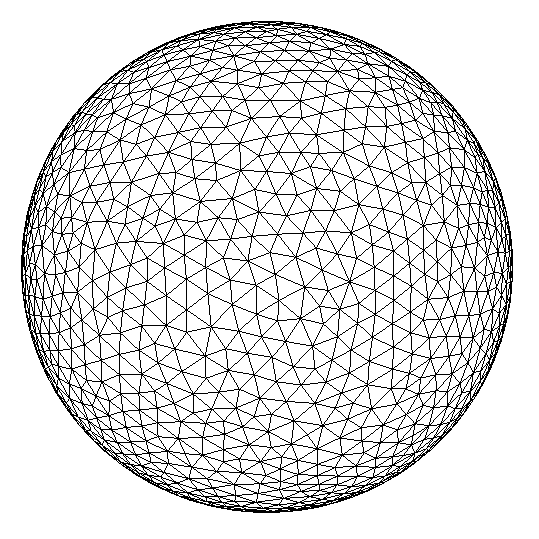}
    \includegraphics[width=0.33\textwidth]{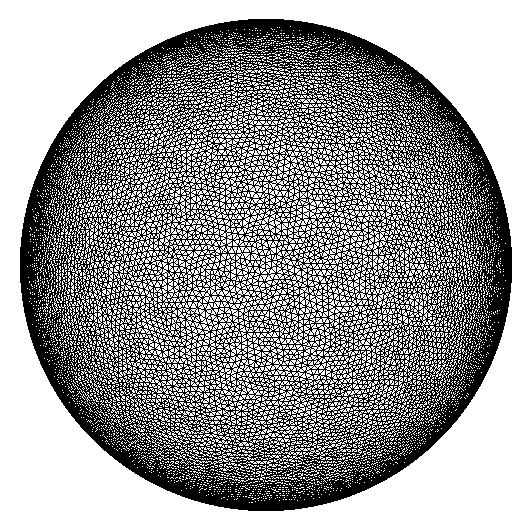}
    \caption{The meshes used for the density method.}
    \label{fig:meshes_density}
\end{figure}

For both optimizations we take $p=20$ and $\eps = 10^{-4}$.

\section{Results : density method.}

In this section we discuss the results obtained by the density method described above. We focus only on the first three eigenvalues, since they already shows a rich behaviour. In each graph, the value of the optimized eigenvalue is plotted in green as a function of the total mass $m$ and the corresponding density is denoted by $\rho^m$. In order to give a comparison, for each $\mu_k$ we plot in red the corresponding eigenvalue for a union of $k$ disjoint geodesic balls of surface area $m/k$ (denoted $\UB^m_k$). This values have been approximated by a finite element (FE) decomposition of the following 1D eigenvalue problem :
\begin{equation}
    \begin{cases}
        -\frac{1}{\sin(\theta)} \frac{\dd}{\dd \theta} \left( \sin(\theta) \frac{\dd y}{\dd \theta} \right) + \frac{1}{\sin^2(\theta)} y = \mu y \mbox{ on } (0, \theta_m) \\
        \frac{\dd y}{\dd \theta}(\theta_m) = 0 \\
        y(0) \mbox{ is finite}
    \end{cases}
\end{equation}
where $\theta_m = \arccos{(1-\frac{m}{2\pi})}$ is the geodesic radius of the ball of surface area $m$ on $\S^2$. In practice, the solution is approximated using $\P_1$ FE with $10000$ degrees of freedom.

\subsection{Validation : $\mu_1$ with constraint.}

In order to validate our optimization process, we rely on the following result of \cite{bucur_sharp_2022}, which states that if we run the optimization outside of a ball of the right area then the optimum is a ball :

\begin{theorem}Let $m \in (0, |\S^n|/2)$ and let $\B^m$ be a geodesic ball of measure $m$ in $\S^n$. Let
$\Omega \subset \S^n \setminus \B^m$ be an open Lipschitz set such that $|\Omega| = m$. Then
$\mu_1(\Omega) \leq \mu_1(\B^m)$.
\end{theorem}

This theorem, proved by some mass transplantation technique "à la Weinberger", is hence also valid for densities and could be reformulated as follows :

\begin{theorem}Let $m \in (0, |\S^n|/2)$ and let $\B^m$ be a geodesic ball of measure $m$ in $\S^n$. Let
$\rho : \S^n \to [0,1]$ such that $\rho = 0$ on $\B^m$ and $ \int_{\S^n} \rho = m$. Then
$\mu_1(\rho) \leq \mu_1(\B^m)$.
\end{theorem}

In practice, we run the optimization process in the whole sphere but add the constraint that all degrees of freedom of $\rho$ that lies inside a certain ball $\B^m$ stay equal to $0$. This constraint is easily handled by IPOPT. The plots in Figures \ref{fig:mu_1_constraint_density_examples} and \ref{fig:mu_1_density_constraint} show that the optimal density is indeed the characteristic function of a ball.

\begin{figure}
    \centering
    \includegraphics[width=0.24\textwidth]{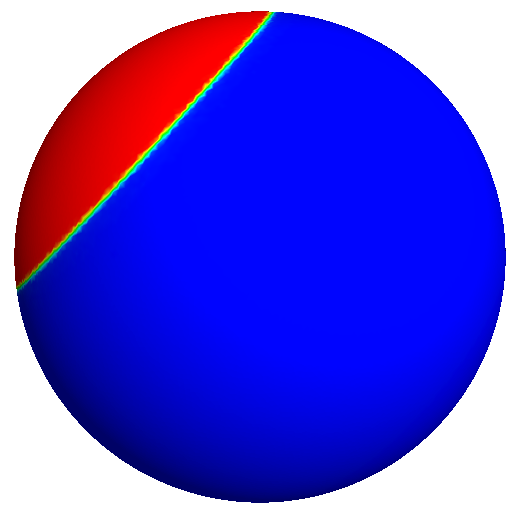}
    \includegraphics[width=0.24\textwidth]{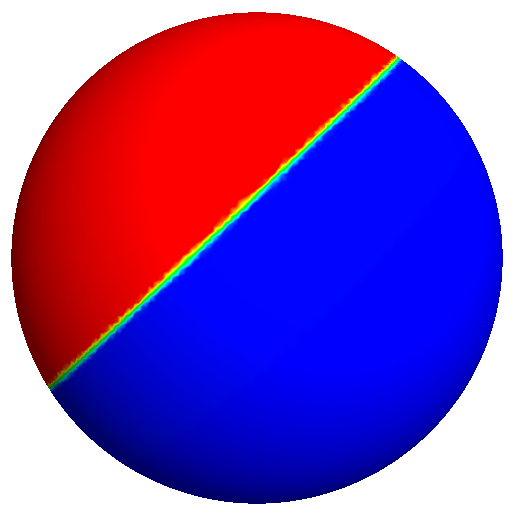}
    \caption{Example of optimal densities for $\mu_1$ for $m \in \{2.17, 5.0\}$ when $\rho$ is supported outside of a ball.}
    \label{fig:mu_1_constraint_density_examples}
\end{figure}

\begin{figure}
    \centering
    \includegraphics[width=0.49\textwidth]{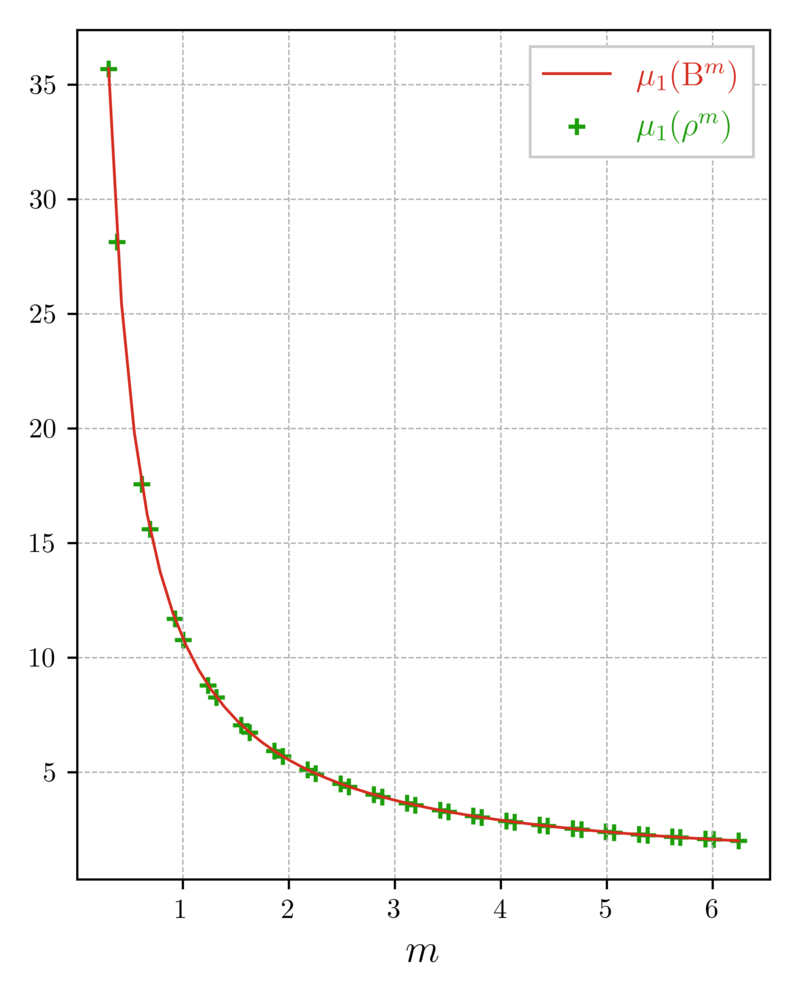}
    \caption{Optimal value of $\mu_1$ obtained by the density method, with the constraint that the support of $\rho$ is located outside of a ball.}
    \label{fig:mu_1_density_constraint}
\end{figure}

\subsection{An interesting case : $\mu_1$ in the whole sphere.}

We now consider the case where $\rho$ is allowed to fill the whole sphere. One first observation is that the optimal eigenvalue is expected to be never less that $n$. Indeed, this eigenvalue is associated to constant densities hence we can chose the density $\rho = \frac{m}{|\S^n|}$ which leads to $\mu_1(\rho) = \mu_1(\S^n) = n$. The simulations actually suggest that this value is only attained near $m=|\S^n|$ as shown in Figure \ref{fig:mu_1_density}.

\begin{figure}
    \centering
    \includegraphics[width=0.49\textwidth]{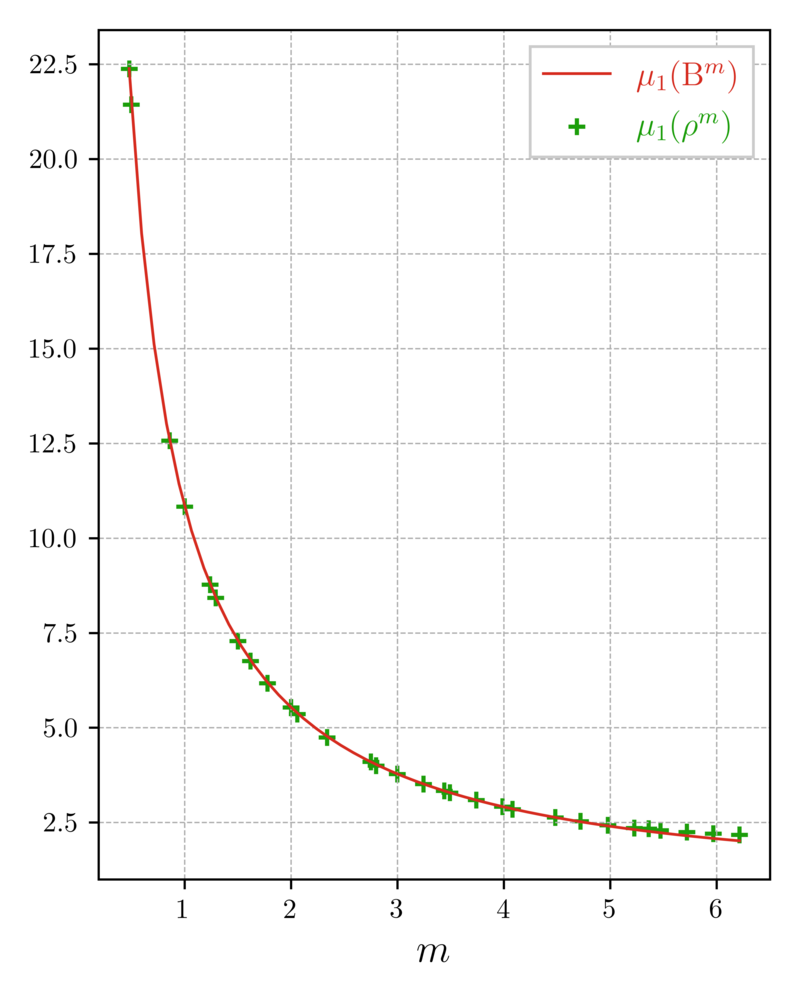}
    \includegraphics[width=0.49\textwidth]{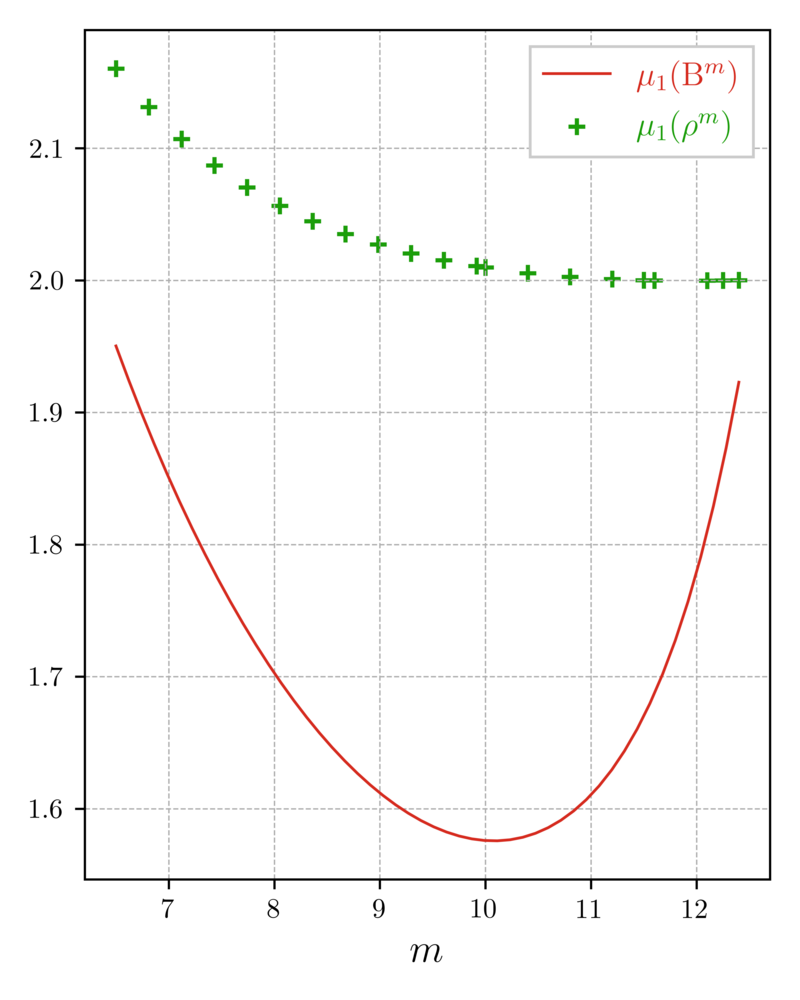}
    \caption{Optimal value of $\mu_1$ obtained by the density method.}
    \label{fig:mu_1_density}
\end{figure}

Since the eigenvalue goes to infinity as $m$ goes to $0$, the graph is displayed on two different scales for a better readability. The reader should pay a particular attention to the range of the different axes. Something interesting happens : the spherical cap seems not to be optimal for values of $m$ greater than $m \approx 4.5$. This is allowed by the fact that $\rho$ can fill the whole sphere, which wasn't allowed with the ball constraint. A zoom on the range $m \in [3.5,6.5]$ is displayed Figure \ref{fig:mu_1_density_split}.

\begin{figure}
    \centering
    \includegraphics[width=0.49\textwidth]{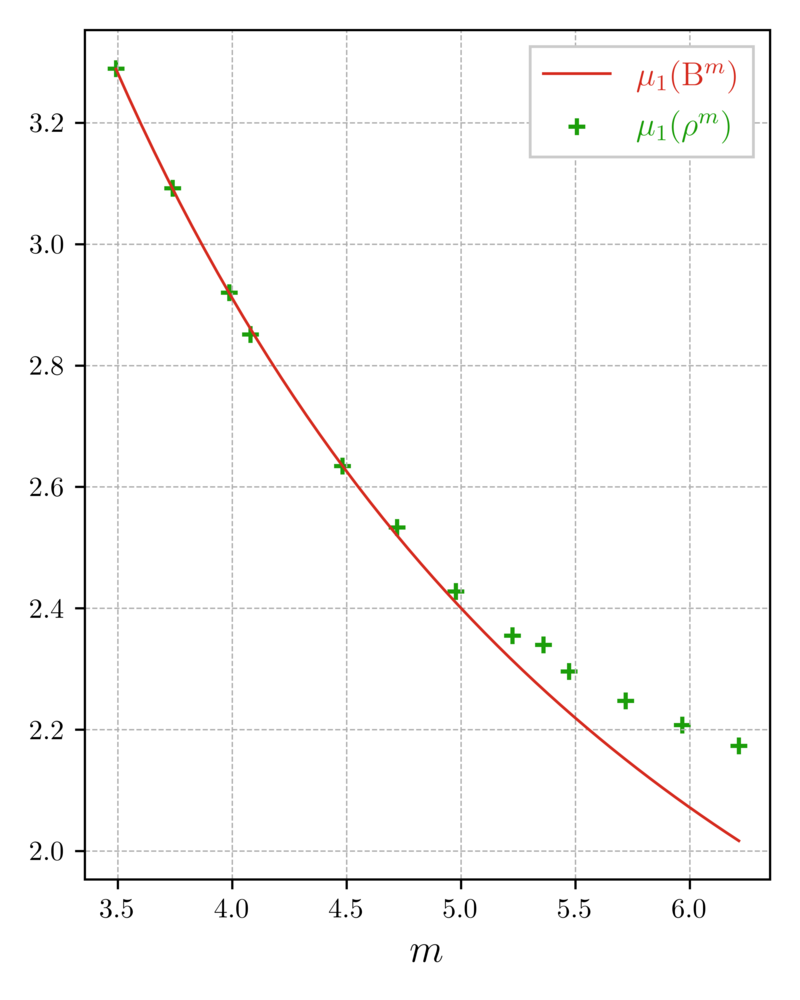}
    \caption{Optimal value of $\mu_1$ near $m\approx 5.0$.}
    \label{fig:mu_1_density_split}
\end{figure}

We illustrate the behaviour of the optimal $\rho$ in Figure \ref{fig:mu_1_density_examples} for different values of $m$. Deep blue color corresponds to $\rho = 0$ while red color corresponds to $\rho = 1$.

\begin{remark}
Apart from $m=2.0$ which seems to be the characteristic function of a geodesic ball, the optimal density seems to be some "homogenized" spherical cap. This surprising result has an important theoretical implication : even if the ball $B^m$ were optimal for $\mu_1$ among shapes $\Omega$ such that $|\Omega| = m \leq |\S^n|/2$,  it would be impossible to prove it using the standard mass transplantation technique of Weinberger \cite{weinberger_isoperimetric_1956}. Indeed, the proof would also hold for densities, but the numerical results strongly indicates that it is false for $m>0.5$. To go further, it would be interesting to investigate if this kind of "homogenization" of the sphere could be attained by some sequence of actual domains. This could for instance suggest the non-existence of optimal domains for a large enough $m$.
\end{remark}

The inspection of these results leads to the following conjecture :

\begin{conjecture}Let $m\in(0,|\S^n|)$. Then the optimal density $\rho^m$ of the problem
\begin{equation*}
  \max \left\{ \mu_1(\rho) :  \rho : \S^n \rightarrow [0,1], \int_{\S^n}\rho =m\right\}.
\end{equation*}
is axially symmetric.
\end{conjecture}

In the light of this conjecture, we illustrate on the same Figure \ref{fig:mu_1_density_examples} the density $\rho$ as the result of the 1D optimization procedure.

\begin{figure}
    \centering
    \includegraphics[width=0.24\textwidth]{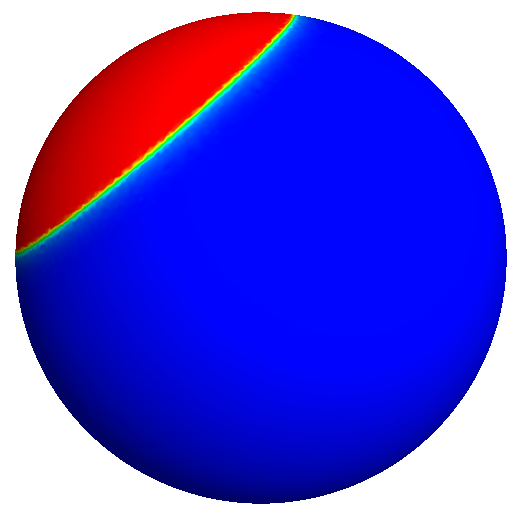}
    \includegraphics[width=0.24\textwidth]{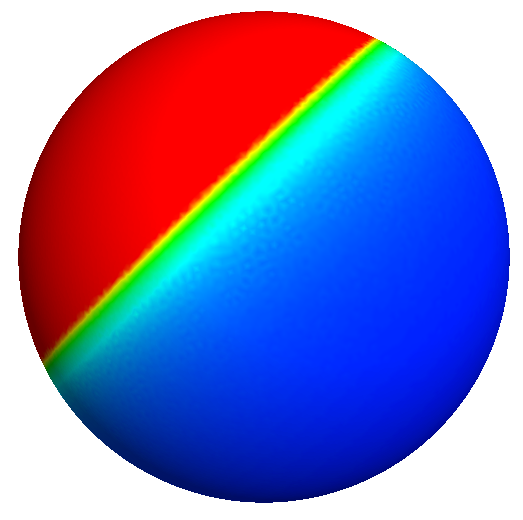}
    \includegraphics[width=0.24\textwidth]{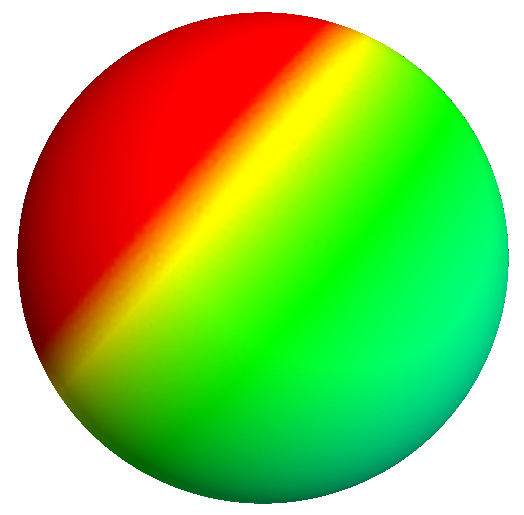}
    \includegraphics[width=0.24\textwidth]{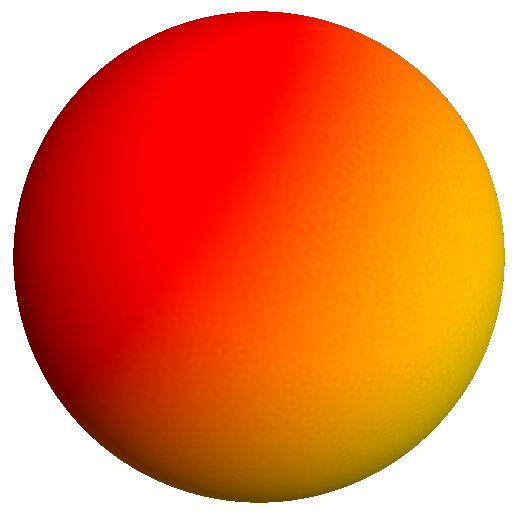}

    \includegraphics[width=0.24\textwidth]{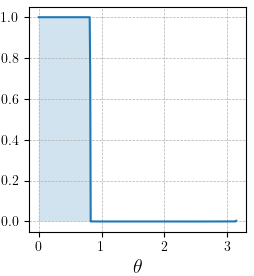}
    \includegraphics[width=0.24\textwidth]{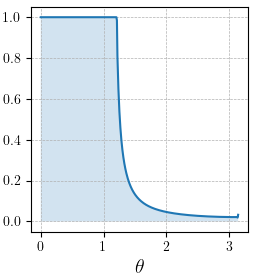}
    \includegraphics[width=0.24\textwidth]{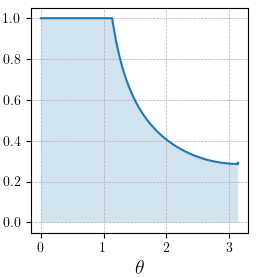}
    \includegraphics[width=0.24\textwidth]{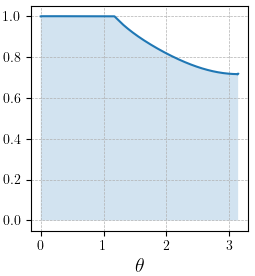}
    \caption{Example of optimal densities for $\mu_1$ for $m \in \{2.0, 4.98, 8.05, 11.2\}$ (top) along with their latitudinal profile (bottom).}
    \label{fig:mu_1_density_examples}
\end{figure}

Note that it would be interesting to get more information on the behaviour of the density near $m=4.5$ where it seems to start to homogenize. A good indication that the optimal density is an actual domain would be that the size of the set $\rho^m \notin \{0,1\}$ is always proportionnal to the size of an element under mesh refinement. On the contrary, if $\rho^m$ is a density, then the size of this set should be independant of the size of one element. Let $N$ be the number of elements of the segment $[0,\pi]$. For $N \in \{100,200,400,800\}$, we compute the quantity
\[
  h_m(N) = \frac{N}{\pi} \int_0^\pi  \rho^m(1-\rho^m).
\]
Since $\pi/N$ is the size of one element, this quantity should be constant in $N$ if $\rho^m$ is a domain for the reasons stated previously. If $\rho^m$ is not a domain, then we would expect $h_m(N)$ to double when doubling the number of points in the mesh.
To compare this quantity for different $N$ we normalize this quantity at $N=100$ and define
\[
  \mbox{Dispertion}_m(N) = \frac{h_m(N)}{h_m(100)}
\]
We plot the graph of $\mbox{Dispertion}_m$ for different values of $m$ in Figure \ref{fig:1d_refinement}

\begin{figure}
    \centering
    \includegraphics[width=0.8\textwidth]{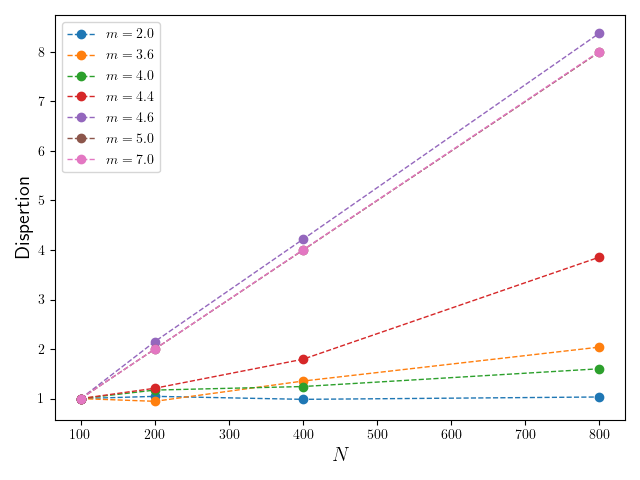}
    \caption{Mesh-refinement procedure une 1D for different values of $m$.}
    \label{fig:1d_refinement}
\end{figure}

A few things can be deduced from this graph. For $m$ large enough (approximately $m>4.6$), the behaviour is the one we would observe for a density which is not a domain, as the dispertion grows with the number of points. On the other hand, it appears that for $m$ small enough $\rho^m$ seems to be a domain since its dispertion is constant. It is the subject of the following conjecture.

\begin{conjecture}There exists $\delta > 0$ such that for all $m \in (0,\delta)$, $\rho^m = \1_{\B^m}$ i.e. the optimal density is the one of a geodesic ball.
\end{conjecture}

Following the numerical observations above, the value of $\delta$ would lie between $3.5$ and $4.6$.

\subsection{Results for $\mu_2$}

\begin{figure}
    \centering
    \includegraphics[width=0.49\textwidth]{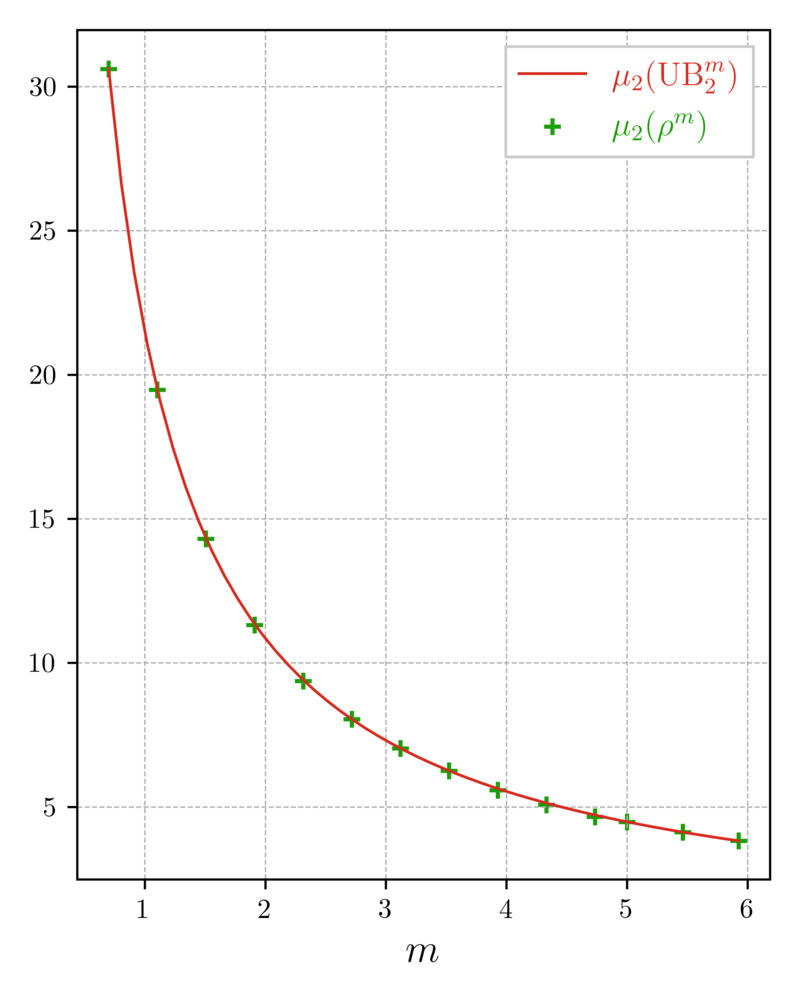}
    \includegraphics[width=0.49\textwidth]{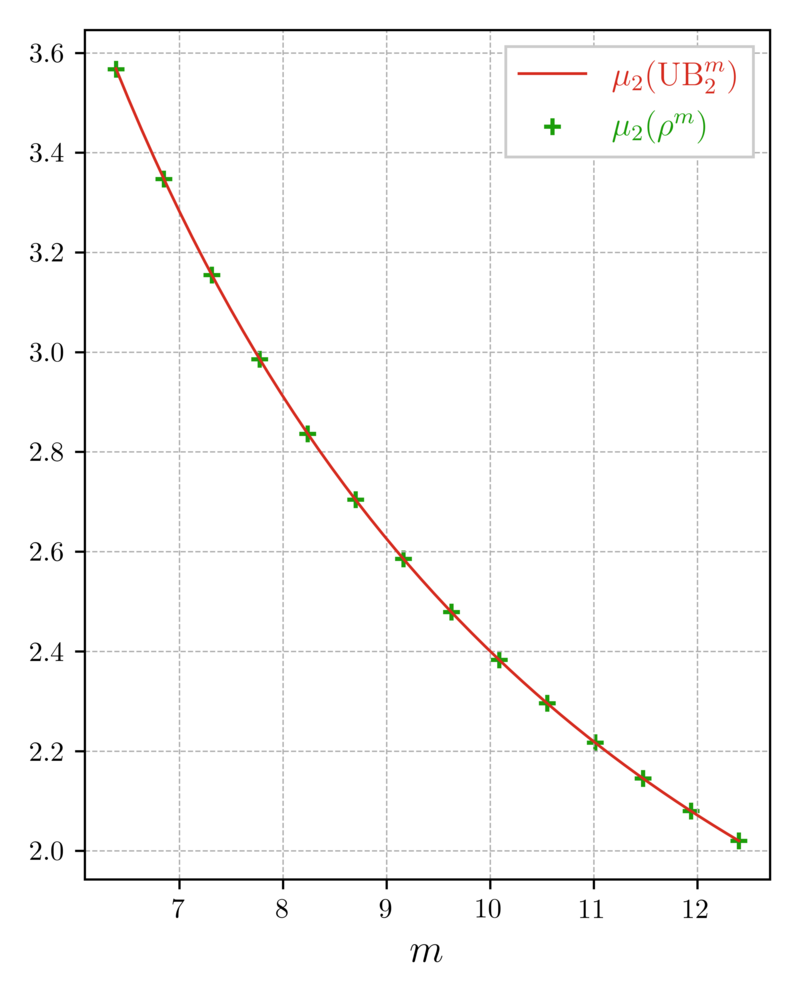}
    \caption{Optimal value of $\mu_2$ obtained by the density method.}
    \label{fig:mu_2_density}
\end{figure}

The results for $\mu_2$ are the most stable ones. Indeed, no matter the value of $m$, the corresponding optimal density is always attained by the characteristic function of the union of two balls of the same measure as can be seen in Figure ~\ref{fig:mu_2_density}. This numerical experiment gives strong insights of the validity of Theorem 2 of \cite{bucur_sharp_2022} which asserts that $\mu_2$ is maximal for the disjoint union of two balls $B^\frac{m}{2} \sqcup B^\frac{m}{2}$. Figure \ref{fig:mu_2_density_examples} shows the optimal densities that are obtained for some values of $m$.

\begin{figure}
    \centering
    \includegraphics[width=0.24\textwidth]{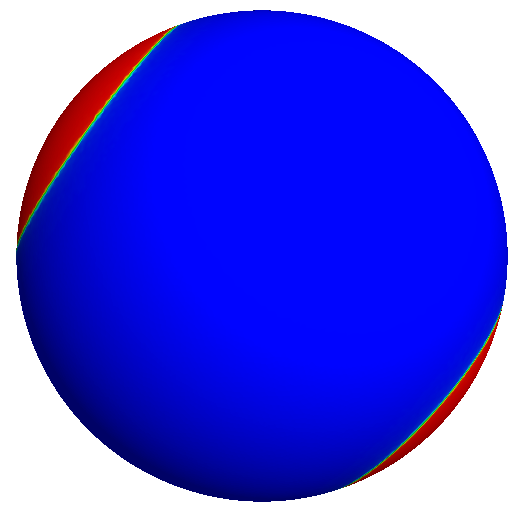}
    \includegraphics[width=0.24\textwidth]{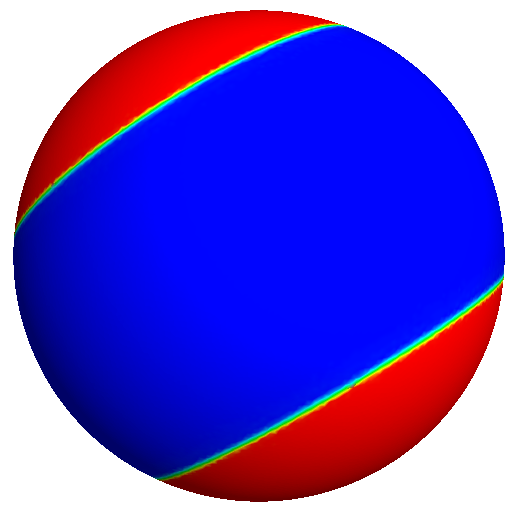}
    \includegraphics[width=0.24\textwidth]{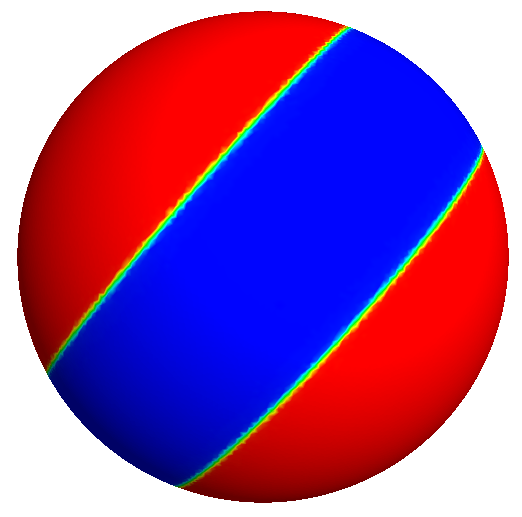}
    \includegraphics[width=0.24\textwidth]{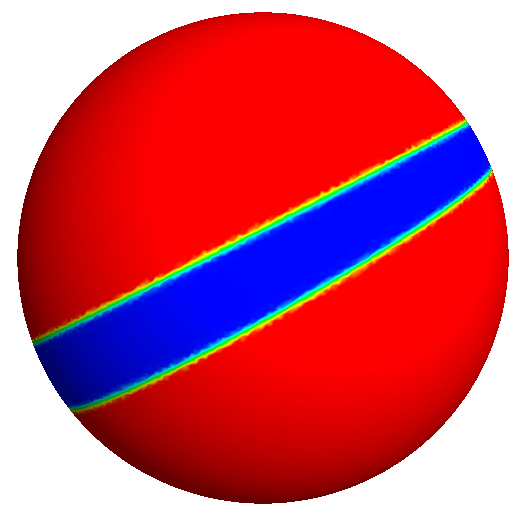}
    \caption{Example of optimal densities for $\mu_2$ for $m \in \{2.31, 5.46, 8.23, 11.01\}$.}
    \label{fig:mu_2_density_examples}
\end{figure}

\subsection{Results for $\mu_3$}

\begin{figure}
    \centering
    \includegraphics[width=0.49\textwidth]{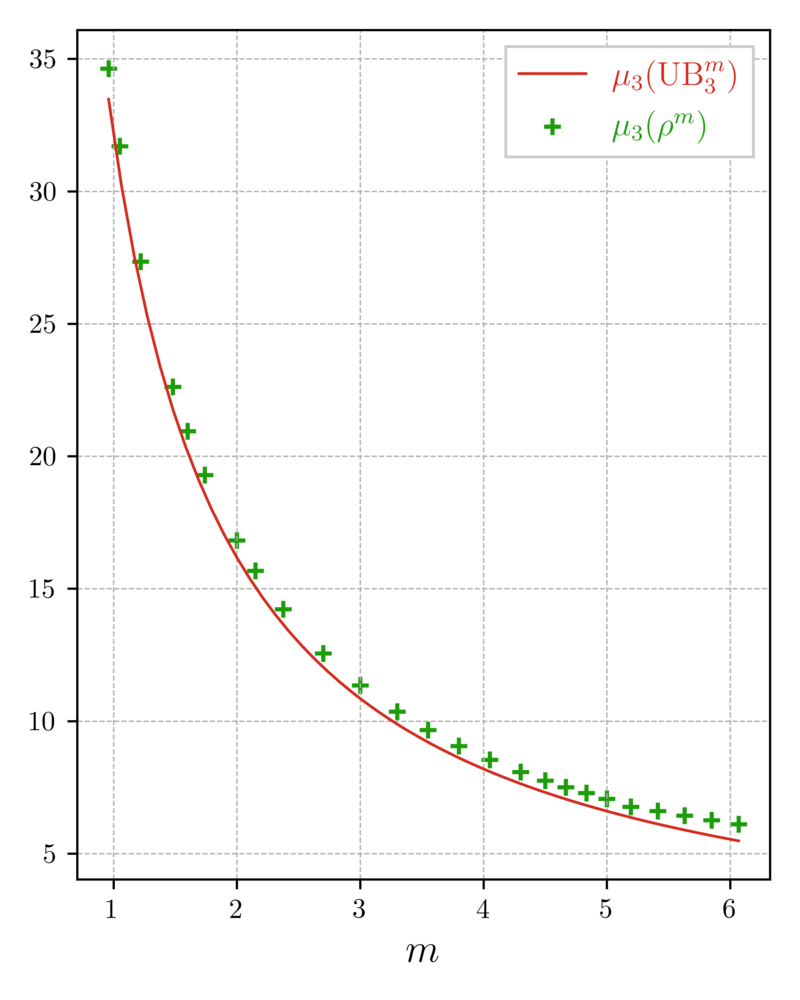}
    \includegraphics[width=0.49\textwidth]{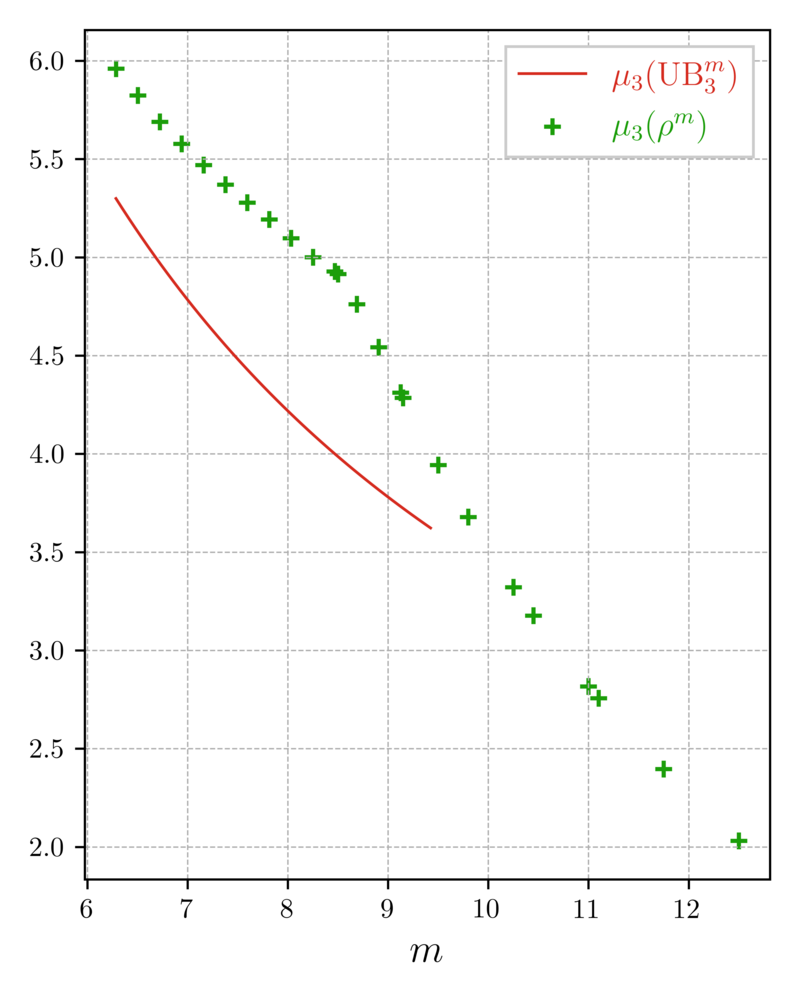}
    \caption{Optimal value of $\mu_3$ obtained by the density method.}
    \label{fig:mu_3_density}
\end{figure}

As for the case of $\mu_1$, the optimization procedure for $\mu_3$ exhibits various different behaviours depending on the value of $m$. Note that the union of three balls of the same surface area seems to never be optimal, as shown in Figure \ref{fig:mu_3_density}.

\begin{figure}
    \centering
    \includegraphics[width=0.24\textwidth]{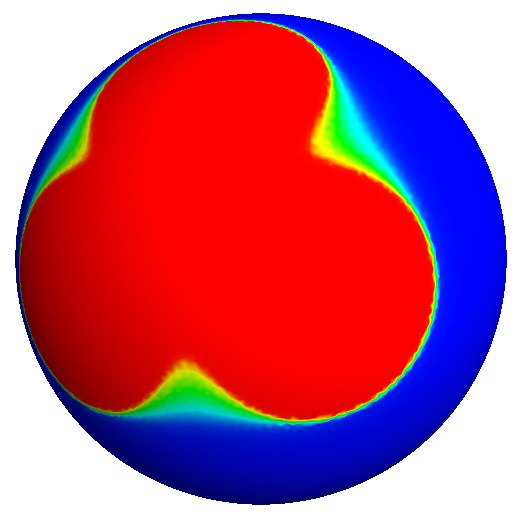}
    \includegraphics[width=0.24\textwidth]{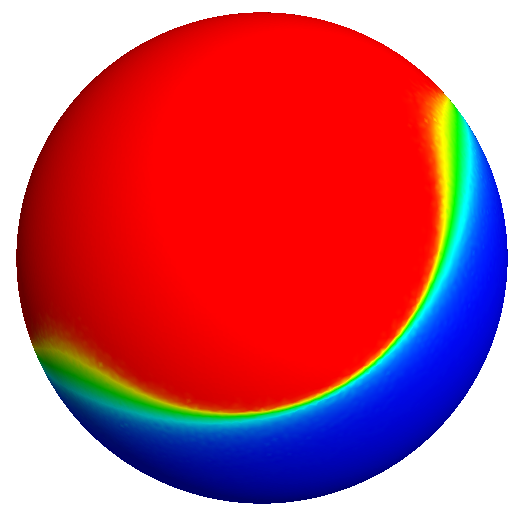}
    \includegraphics[width=0.24\textwidth]{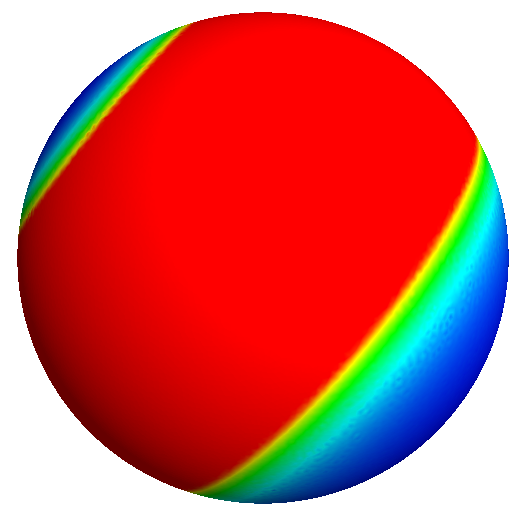}
    \includegraphics[width=0.24\textwidth]{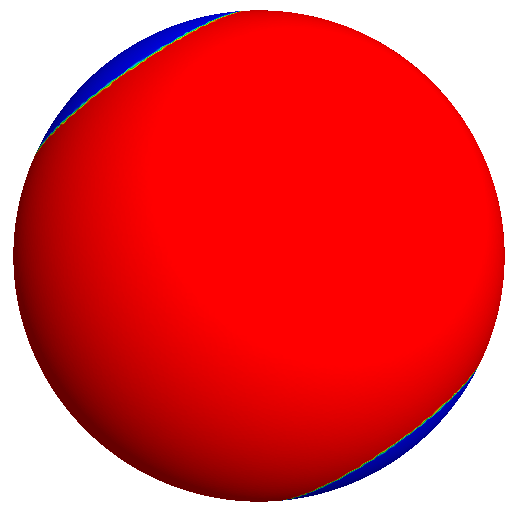}
    \caption{Example of optimal densities for $\mu_3$ for $m \in \{2.0, 5.0, 8.03, 11.0\}$.}
    \label{fig:mu_3_density_examples}
\end{figure}

In Figure \ref{fig:mu_3_density_examples} are displayed the different types of densities that can be obtained with the density method.  As it could be expected, for small $m$ we get the same type of result than in the plane \cite{BMO_2022}. Only the last, for large $m$, seems to be an actual characteristic function of some "napkin ring"-shaped domain. For $m \approx 8.0$, we get some homogenized geodesic annulus. On the two-dimensional sphere it is expected to see that for large $m$ the eigenvalue goes to $2$ since $\mu_1(\S^2)$ is of multiplicity $3$ and $\mu_1(\S^2)=2$.

\subsection{Data.}

All the final solutions are available in MEDIT format at \url{https://github.com/EloiMartinet/Neumann_Sphere/}. A FreeFem++ \cite{MR3043640} script allowing to read the solutions and compute the eigenvalues is also provided for replicability purposes. See the README file for more information.

\section{Level-set method}

In this section we focus on the optimization of the original problem (\ref{domain_problem}) through the level-set method. This allows to give more informations for the original shape optimization problem, when optima lead by the previous density method didn't match the characteristic function of an actual shape.
The level-set method consists in representing the domain $\Omega$ as the level-set of a function has been extensively used for shape optimization, either for compliance or eigenvalue optimization (see for example \cite{de_gournay_velocity_2006}, \cite{allaire_level-set_2005}, \cite{allaire_shape_2014} among others). In the line of \cite{allaire_shape_2014} and in order to fix the notations, we recall main ideas of the level set approach.

Let $t\in[0,T]$ and $\Omega(t) \subset \S^n$ be a domain evolving in time according to a velocity field $V : [0,T]\times\S^n \to \mathbb{T}\S^n$. More precisely, if $\Omega_0$ is a domain in $\S^n$ and $\chi$ is the flow of $V$ defined as the solution of the differential equation
\begin{equation*}
    \begin{cases}
    \chi'(x_0,t)= V(t,\chi(x_0,t)) \\
    \chi(x_0,0)= x_0 \in \Omega_0
    \end{cases}
\end{equation*}
then we can define
\[
\Omega(t) := \chi(\Omega_0,t).
\]

In the framework of the level-set method, we represent our domain by a function $\phi : [0,T] \times \S^N \to \R$ such that
\begin{equation}
    \forall x \in \S^N, \forall t\in[0,T],
    \begin{cases}
    \phi(t,x) < 0 & \mbox{ if } x \in \Omega(t) \\
    \phi(t, x) = 0 & \mbox{ if } x \in \partial \Omega(t) \\
    \phi(t,x) > 0 & \mbox{ if } x \in \S^N \sm \Omega(t)
    \end{cases}.
\end{equation}

The motion of $\Omega(t)$ is equivalent to the advection of $\phi$ by the equation
\begin{equation}
    \partial_t \phi(t,x) + V(t,x) \cdot \nabla \phi(t,x) = 0 \mbox{ on } (0,T)\times \S^n.
\end{equation}

Observing that $n_{\Omega(t)}:=\frac{\nabla \phi(t,.)}{|\nabla \phi(t,.)|}$ is an extension of the unitary outward normal of $\Omega(t)$ for all $t$ and if $V$ is of the kind $V(t,x) = v(t,x)n_{\Omega(t)}(x)$ with $v(t,x)\in \R$, we can re-write the previous equation as

\begin{equation}
     \partial_t \phi(t,x) + v(t,x) |\nabla \phi(t,x)| = 0 \mbox{ on } (0,T)\times \S^n.
\end{equation}

In our case, $[0,T]$ represents a single timestep and therefore $T$ is small enough to consider that $v(t,x) \approx v(x)$ for all $t \in [0,T]$ which finally leads to the so called \textit{Hamilton-Jacobi} equation :

\begin{equation}
    \label{eqn:advection}
     \partial_t \phi(t,x) + v(x) |\nabla \phi(t,x)| = 0 \mbox{ on } (0,T)\times \S^n.
\end{equation}

Numerically, this equation is solved by the method of characteristics, thanks to the \textit{"advect"} toolbox that can be found at \url{https://github.com/ISCDtoolbox/Advection/}. This method supports $\P_1$ functions defined on surface meshes. Many thanks to the authors for their precious work.

\subsection{Shape derivative}

At each time step, the purpose is to find some vector field $v$ such that advecting $\phi$ by (\ref{eqn:advection}) on this small time step increase the considered eigenvalue. In order to address this issue we need to compute the shape derivative of an eigenvalue of a domain of $\S^n$. This shape derivative is given by the following result (see for instance \cite{ZANGER200139}):

\begin{theorem}[Shape derivative] We assume that $\Omega$ is $C^1$ with non-empty boundary. Let $k \in \N$ and $V : \S^N \to \mathbb{T}\S^n$ be a smooth vector field with compact support in the neighborhood of $\Omega_0$. We denote
\[
\mu_k'(\Omega_0, V) := \lim_{t \to 0^+} \frac{\mu_k(\Omega(t))-\mu_k(\Omega_0)}{t}.
\]
Moreover we assume that $\mu_k(\Omega_0)$ is simple. Then this limit exists and
\begin{equation}
    \label{eqn:shape_derivative}
    \mu_k'(\Omega_0, V) = \int_{\partial \Omega_0} \left(|\nabla u|^2 - \mu_k(\Omega_0) u^2\right)(V.n) \dd \sigma
\end{equation}
with $u$ an eigenfunction associated to $\mu_k(\Omega_0)$ with unitary $\L^2$ norm and $n$ the outward normal.
\end{theorem}

The previous theorem allows us to only consider normal variations of the boundary and from this consideration $\mu_k'(\Omega_0,V) = \mu_k'(\Omega_0,v)$ where $v = V.n$. Moreover, it shows that we can choose $v=|\nabla u|^2 - \mu_k(\Omega_0)u^2$ as a gradient direction. This, however, may not be the best choice as we see hereafter.

\subsection{Handling area constraint}

Despite being able to compute a gradient direction (\ref{eqn:shape_derivative}), we also have to fullfill the constraint $|\Omega| = m$ in the original problem (\ref{domain_problem}). We could choose to add a penalization term and maximize the function
\[
\Omega \mapsto \mu_k(\Omega) - b (|\Omega| - m')^2
\]
instead, with $b>0$ and $m'$ a parameter allowing to control for the total mass $m$. However, if we consider the sequence of geodesic balls $B(\eps)$ of radius $\eps>0$ then we can fin that
\[
\mu_k(B(\eps)) - b (|B(\eps)|-m')^2 \xrightarrow[\eps\to 0]{} +\infty
\]
which establishes that the optimum is never attained for a positive area.
However, a result by Strichartz allows to replace the unbounded quantity $\mu_k(\Omega)$ by quantity $|\Omega|^\frac{2}{n} \mu_k(\Omega)$ even if we do not have invariance by dilation as in the Euclidian case. Indeed, using formulas (3.15) and (3.16) of \cite{strichartz_estimates_1996} in the special case $n=2$, we get the following property :

\begin{proposition}
Let $\Om \subset \S^2$. Then
\begin{equation} \label{ineq:strichartz_bound}
    |\Om|\mu_k(\Om) \leq 2\pi k^2.
\end{equation}
\end{proposition}

This generic bound allows to maximize the function
\begin{equation}
    \label{eqn:cost}
    J(\Omega) := |\Omega| \mu_k(\Omega) - b (|\Omega| - m')^2
\end{equation}
which prevents the function to blow-up. Then if $\Omega^*$ maximizes (\ref{eqn:cost}), it is solution of (\ref{domain_problem}) with $m=|\Omega^*|$.

There is two way to implement the level set method. The so-called ersatz material approach involves a fixed mesh where the "void" part is filled with some material with good properties. The other one involves to remesh the domain at each step according to the level-set function. While the second one is more accurate, it suffers two main drawbacks, the most obvious one being its computational cost. The second one is related to the connectivity of $\Omega(t)$: suppose that we want to optimize $\mu_k$, starting from a topologically complex domain. The level-set method allowing topological changes, it is very likely that at one point $t$, $\Omega(t)$ splits into $k+1$ connected components. Then $\mu_0(\Omega(t))= ... = \mu_k(\Omega(t)) = 0$ and the associated eigenfunctions are constant on each connected component. This leads the shape derivative to be equal to
\[
  \mu_k'(\Omega,v) = - \mu_k(\Omega) \int_{\partial \Omega} v \dd \sigma.
\]
The reader may recognise that this is proportionnal to the shape derivative of the function $\Om \to |\Om|$. This implies that the optimization process will only optimize on the volume. One the other hand, the "ersatz material" approach allows transparent topological changes and is faster than the second one since it doesn't require remeshing at each iteration. This is why we perform a first optimization using the ersatz material method and then use a remeshing approach for a final optimization of higher accuracy.

\subsection{Level-set with ersatz material}

In this section we assume the eigenvalues to be simple. According to the approximation theorem \ref{th:approximation}, we fix a small $\eps>0$ and solve the problem
\[
 -\div[(\1_\Omega+\eps)\nabla u] = \mu_k^\eps(\1_\Omega) (\1_\Omega+\eps^2) u
\]
at each step. Thanks to the above-mentioned theorem, we expect this eigenvalue to be close to the actual one for small $\eps$.
However, the function $\1_\Omega = 1_{\{\phi < 0 \}}$, as defined on the mesh, may be highly irregular. This is why we approximate it in the following way
\[
  \1_\Omega \approx \frac{1}{2}\left( 1- \frac{\phi}{\sqrt{\phi^2 + \sigma^2}} \right)
\]
with $\sigma>0$ small. This avoids degeneracy in the denominator. Similar regularizations have been used in this framework, see for instance \cite{allaire2014multi} For the same reasons, the extended normal field is approximated by
\[
n_\Omega \approx \frac{\nabla\phi}{\sqrt{|\nabla \phi|^2 + \sigma^2}}
\]

\subsection{Initialization}

It is well-known that the levelset method is prone to fall into local optima because of its sensitivity to initialization. To tackle this problem, the levelset function is initilized with a randomized trigonometric sum of the type
\begin{equation*}
  \phi(\theta, \psi) = \Re \left\{\sum_{j=0}^{p} \sum_{k=0}^q c_{j,k} \exp{i(j\theta + k\psi)}\right\}
\end{equation*}
where the $c_{i,j}$ are chosen at random and $\theta,\psi \in [0,\pi]\times[0,2\pi]$ are respectively the latitude and longitude on $\S^2$. It is expected that the larger $q$ and $p$, the more complex $\phi$ is and, by extension, $\Omega$.

\subsection{Multiple eigenvalues}

The case of multiple eigenvalues, which always occursin practice, is handled in the same way as in the density case.

\subsection{Numerical considerations}

In our simulations, we took $\eps = 10^{-4}$ and $\sigma = 10^{-5}$. Moreover, to capture the variations of $\Omega(t)$ with good accuracy and because the level-set function tends to steepen near $\partial \Omega(t)$ over the iterations, we remesh the domain thanks to the MMG library \cite{dapogny2014three} and recompute the signed distance function every $20$ iterations thanks to the mshdist tool \cite{dapogny2012computation}. The maximal size of an element is $h_{max} = 10^{-1}$ and the minimal size is $h_{min} = 10^{-3}$. Just as previously, we use $\P_1$ finite elements and the FE computations are performed in GetFEM. The optimization algorithm is a simple gradient algorithm with a fixed number of $N=600$ steps. The step size $\delta t$ is chosen such that, if $v$ is the gradient direction at a given moment then $\delta t = \frac{\gamma}{\|v\|_{\infty}}$ with $\gamma=3.10^{-2}$. The penalty term $b$ is chosen to be equal to $5$ and $m'$ takes multiple values between $0.5$ and $4\pi$. The algorithm is presented in algorithm \ref{alg:ersatz} for clarity purposes. Note that this algorithm is the one performed for a fixed $m'$. Finally, as in the density case, the optimization is performed multiple times with different initial level set functions and the best one is kept.

\begin{algorithm}
\caption{ersatz material levelset algorithm.}\label{alg:ersatz}
\begin{algorithmic}
  \Require $k > 0$ \Comment{The eigenvalue we optimize}
  \Require A mesh in MEDIT format
  \State Initialize the levelset $\phi$
  \For{$i$ from $0$ to $N$}
    \If{$N=0 \mod 20$}
        \State Remesh using MMG.
        \State Reinitialize the levelset function $\phi$.
    \EndIf
    \State Compute $\mu_k^\eps(\Omega)$ and the associated eigenfunctions using FE.
    \State Compute $v$ the maximizing direction.
    \State Advect $\phi$ during a time $\delta t$.
  \EndFor
\end{algorithmic}
\end{algorithm}

\subsection{Level-set with remeshing}

The following method is triggered once the previous one has converged, hence we don't expect major changes in topology which could be problematic as discussed before. In this procedure, we remesh the sphere  such that $\partial \Omega = \{\phi = 0\}$ is a polygonal line of the mesh at each timestep. This then allows us to extract the mesh describing $\Omega$ and solve the original eigenvalue problem on it, without having to compute the approximation $\mu_k^\eps$. But then the optimization direction $v = |\nabla u|^2 - \mu_k(\Omega)u^2$ (with $u$ an eigenfunction of $\mu_k(\Omega)$) is only defined on $\Omega$ while we need it to be defined in the whole sphere $\S^n$ in order to advect the level-set function. In this purpose we use the well-known "extension-regularization" method which allows - as its name suggests - to extend the velocity field on all $\S^n$ and regularize it at the same time \cite{de_gournay_velocity_2006}. Still assuming that the eigenvalue is simple, we see that $v \mapsto \mu_k'(\Omega_0,v)$ is a continuous linear form in $v$. Hence, we can find $w$ the unique solution to the variational problem
\begin{equation}
    \forall v \in \H^1(\S^n), \int_{\S^n} \alpha \nabla v \nabla w + vw = \dd \mu_k(\Omega_0, v)
\end{equation}
where $\alpha > 0$. Then $w$ is indeed an extension of $|\nabla u|^2 - \mu_k u^2$ on $\H^1(\S^n)$ with regularity depending on $\alpha$. Moreover, $w$ is a valid gradient direction since
\[
\dd \mu_k(\Omega_0, w) = \int_{\S^n} \alpha |\nabla w|^2 + w^2 \geq 0.
\]

\subsection{Numerical considerations}

The numerical values chosen for this second optimization are mostly the same as the previous procedure. One add the regularization parameter $\alpha=0.1$ and that $h_{max}$ is now equal to $5.10^{-2}$. The step size $\delta t_i$ is now adaptative : if at a given iteration $i$ we have that $\mu_k(\Omega_i) > \mu_k(\Omega_{i+1})$ then $\delta t_{i+1} = \delta t_{i}/2$. Otherwise $\delta t_{i+1} = 1.1 \delta t_i$. The optimization stops when $\delta t_i < 10^{-7}$ and the mesh with the best cost is kept. The pseudocode of the algorithm for a fixed $m'$ is provided in algorithm ~\ref{alg:remesh}.

\begin{algorithm}
\caption{ersatz material levelset algorithm.}\label{alg:remesh}
\begin{algorithmic}
  \Require $k > 0$ \Comment{The eigenvalue we optimize}
  \Require The mesh and optimal domain $\Omega$ obtained by the previous procedure.
  \State Initialize the levelset $\phi$ as the one of $\Omega$.
  \While{$\delta t > 10^{-7}$}
    \State Remesh using MMG.
    \State Reinitialize the levelset function $\phi$.
    \State Extract the submesh $\Omega$
    \State Compute $\mu_k(\Omega)$ and the associated eigenfunctions using FE.
    \State Compute $v$ the maximizing direction on $\partial \Omega$
    \State Extend $v$ to the whole mesh of $\S^2$ by extension-regularization
    \State Advect $\phi$ during a time $\delta t$
    \If{The cost function increased}
      \State $\delta t \gets 1.1 \delta t$
    \Else
    \State $\delta t \gets \delta t/2$
    \EndIf
  \EndWhile
\end{algorithmic}
\end{algorithm}

\section{Results : level-set method.}

We report here the optimization results for $k \in \{1,2,3\}$. We denote by $\Omega^m$ the optimum computed with the levelset procedure verifying $|\Omega^m|=m$. The optimal eigenvalues $\mu_k(|\Omega^m|)$  are plotted in green, against the corresponding surface area $m$. As for the density method, we also plot in red the eigenvalue $\mu_k\left(UB(|\Omega^m|, k)\right)$ of an union of $k$ disjoints balls of total area $m$. Since the eigenvalue goes to $\infty$ as $|\Omega^m|$ goes to $0$, we divide the plot on two parts $0<|\Omega^m| \leq 2\pi$ and $2\pi < |\Omega^m| \leq 4\pi$ for better readability.

\subsection{Optimization of $\mu_1$}

\begin{figure}
    \centering
    \includegraphics[width=0.49\textwidth]{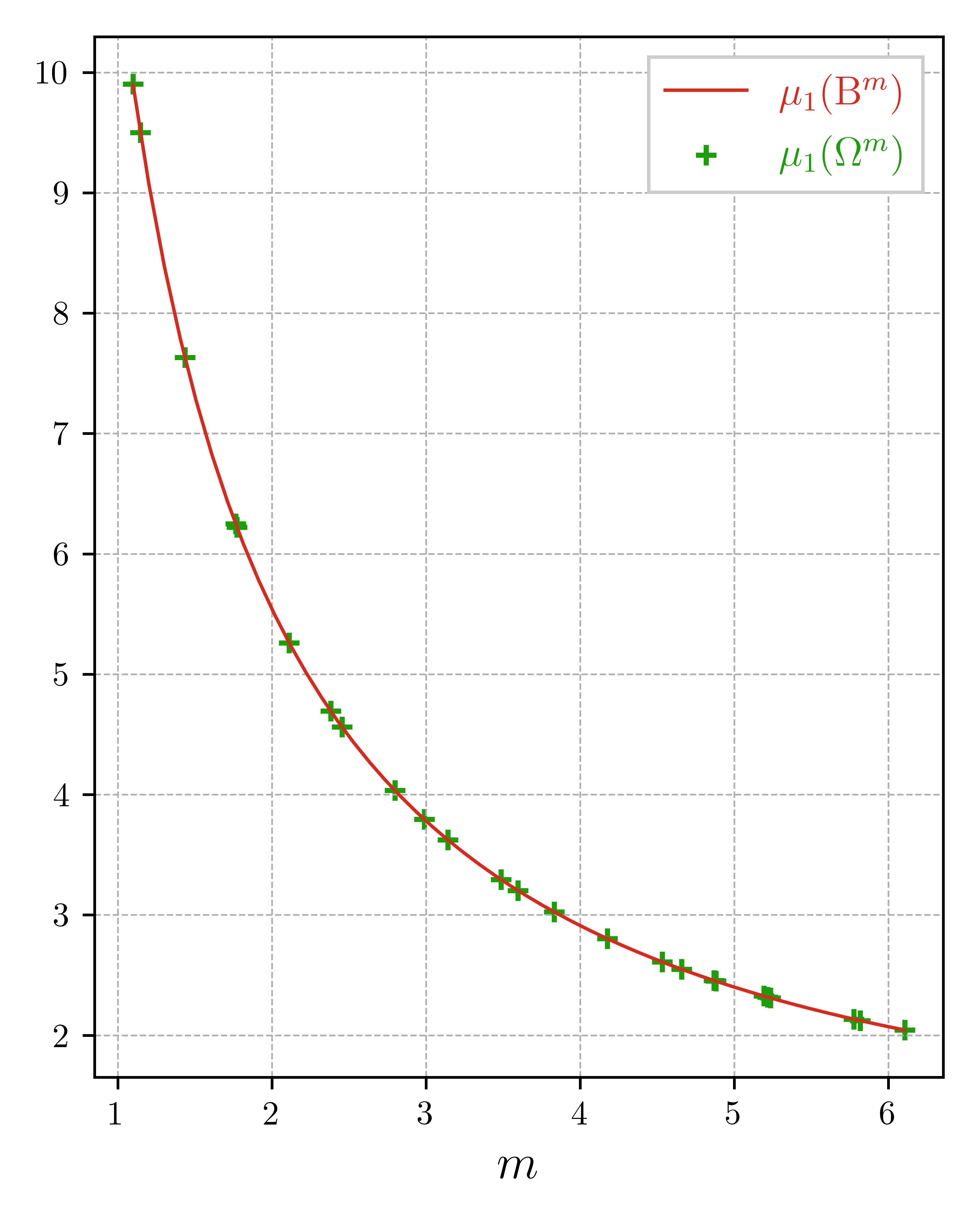}
    \includegraphics[width=0.49\textwidth]{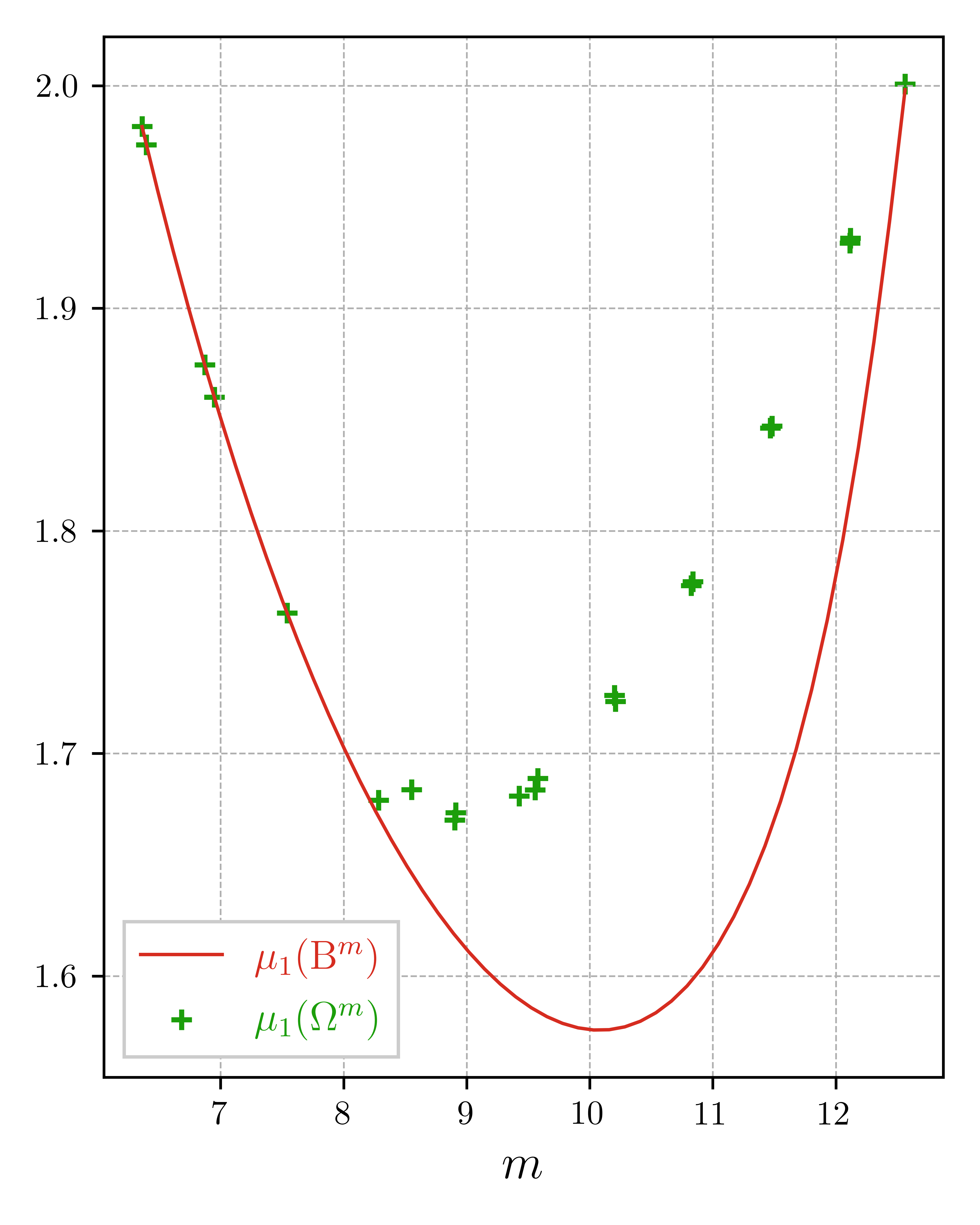}
    \caption{Optimal value of $\mu_1$ obtained by the level-set method.}
    \label{fig:mu_1_ls}
\end{figure}

In Figure \ref{fig:mu_1_ls} are displayed the results for the optimization of $\mu_1$. The spherical cap seems to be the optimal shape up to $m \approx 8.0$, after which it clearly ceases to be the optimal shape. From that point up to $m=4\pi$, complex shapes arises, consisting in a plain hemisphere and a lot of holes in the opposite one. Different views of one of those shapes can be seen in Figure \ref{fig:mu_1_views}, where $m \approx 11.13 $ and $\mu_1(\Omega^m) \approx 1.77$ (for instance, a spherical cap of this surface area would give $\mu_1(\B^m) = 1.62$). This strange behaviour, combined with the density approach above, may suggest that the actual optimal may be attained by some kind of homogenization procedure. Some simple, non conclusive numerical test have been performed in this direction but this problem surely needs further investigation and may lead to interesting numerical and theoretical developements.

\begin{figure}
    \centering
    \includegraphics[width=0.24\textwidth]{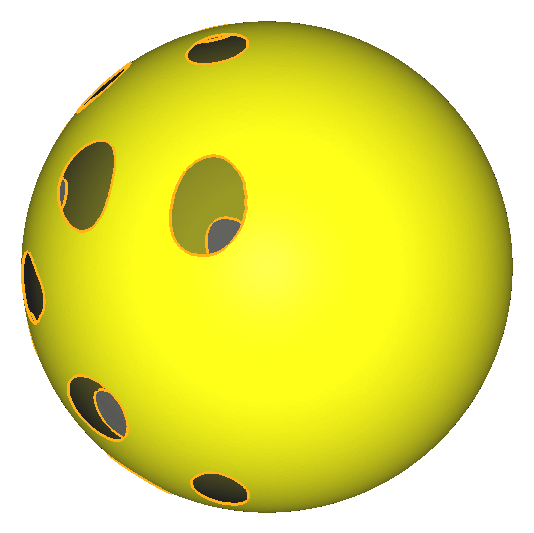}
     \includegraphics[width=0.24\textwidth]{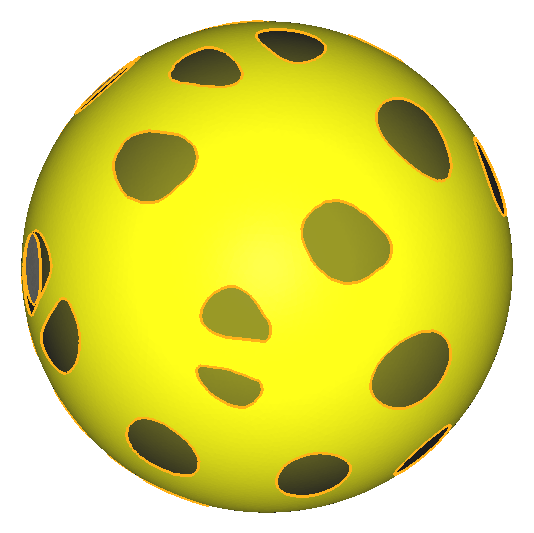}
      \includegraphics[width=0.24\textwidth]{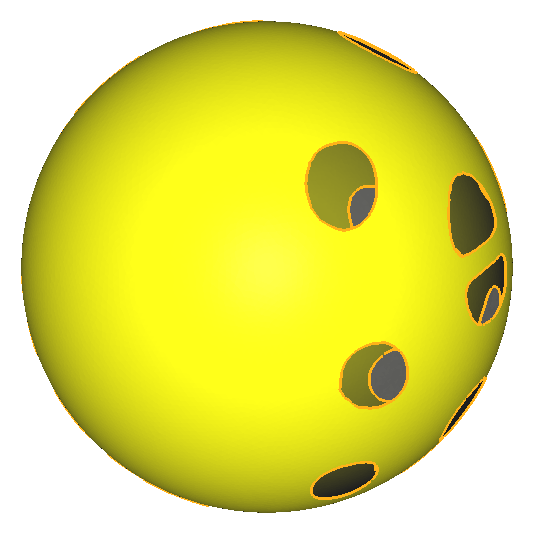}
       \includegraphics[width=0.24\textwidth]{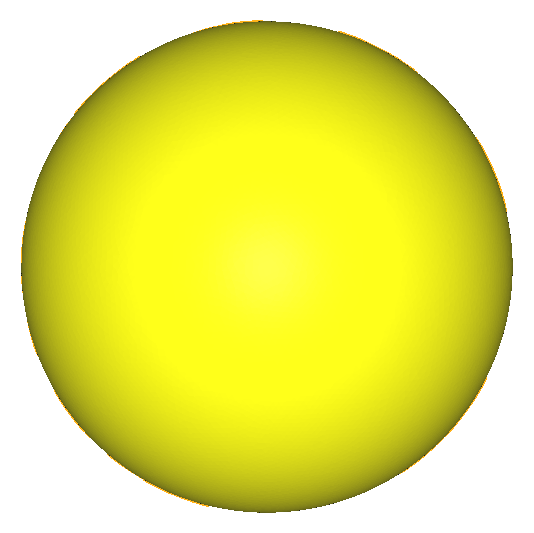}
    \caption{Rotationnal view of the optimal shape obtained by the level-set method for large $m$.}
    \label{fig:mu_1_views}
\end{figure}

More optimal shapes are displayed in Figure \ref{fig:mu_1_levelset_examples}. Looking at $\Omega^m$ for $m=8.0$ (the third one from the left), one can imagine that it would be possible for the geodesic cap to cease to be optimal for $m$ way lower than $8.0$ but the numerical procedure wouldn't be able to "see" it because it would be necessary to create details thinner than the size of an element of the mesh. However, it seems unlikely that the spherical cap ceases to be optimal for the same mass $m$ as the density method:

\begin{conjecture}Let $\delta$ be the same as in Conjecture 2. Then there exists $\delta'>\delta$ such that for all $0 < m < \delta'$, $\Omega^m=B^m$
\end{conjecture}

\begin{remark}The fact that $\delta' \geq \delta$ is obvious; the interesting part would be to show that the inequality is strict.
\end{remark}

\begin{figure}
    \centering
    \includegraphics[width=0.24\textwidth]{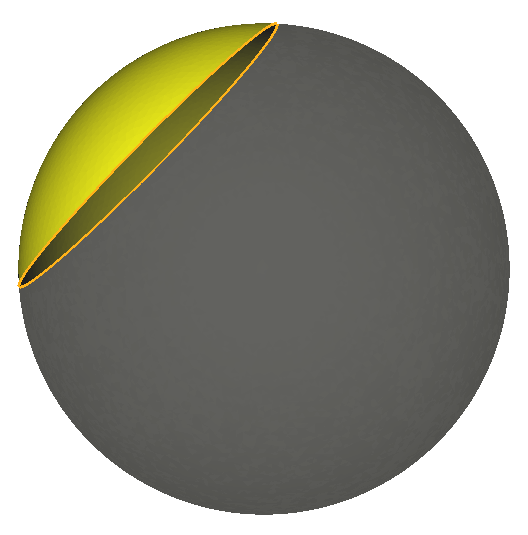}
    \includegraphics[width=0.24\textwidth]{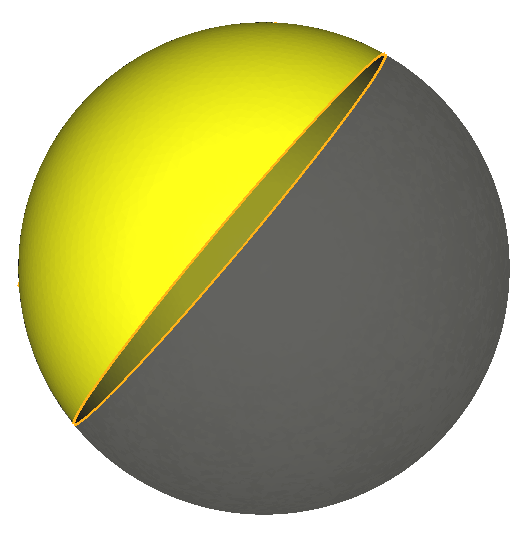}
    \includegraphics[width=0.24\textwidth]{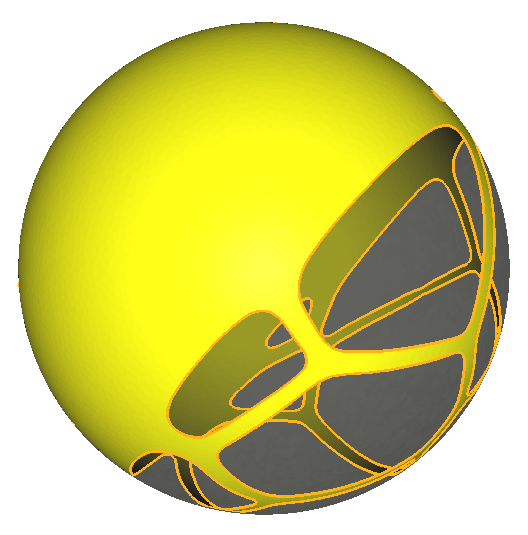}
    \includegraphics[width=0.24\textwidth]{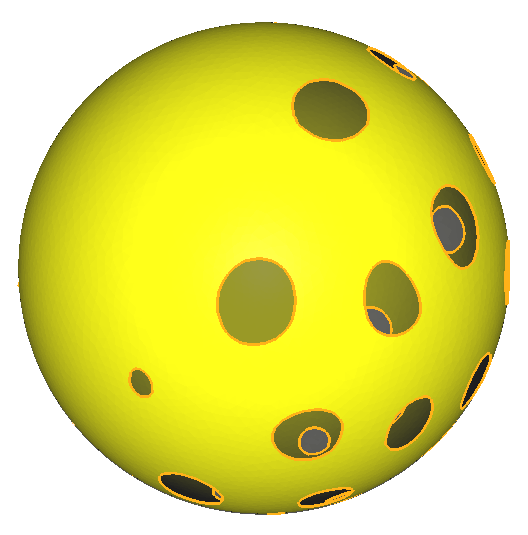}
    \caption{Example of optimal domains for $\mu_1$ for $m \in \{2.03, 5.1, 8.0, 10.85\}$.}
    \label{fig:mu_1_levelset_examples}
\end{figure}

\subsection{Optimization of $\mu_2$}

\begin{figure}
    \centering
    \includegraphics[width=0.49\textwidth]{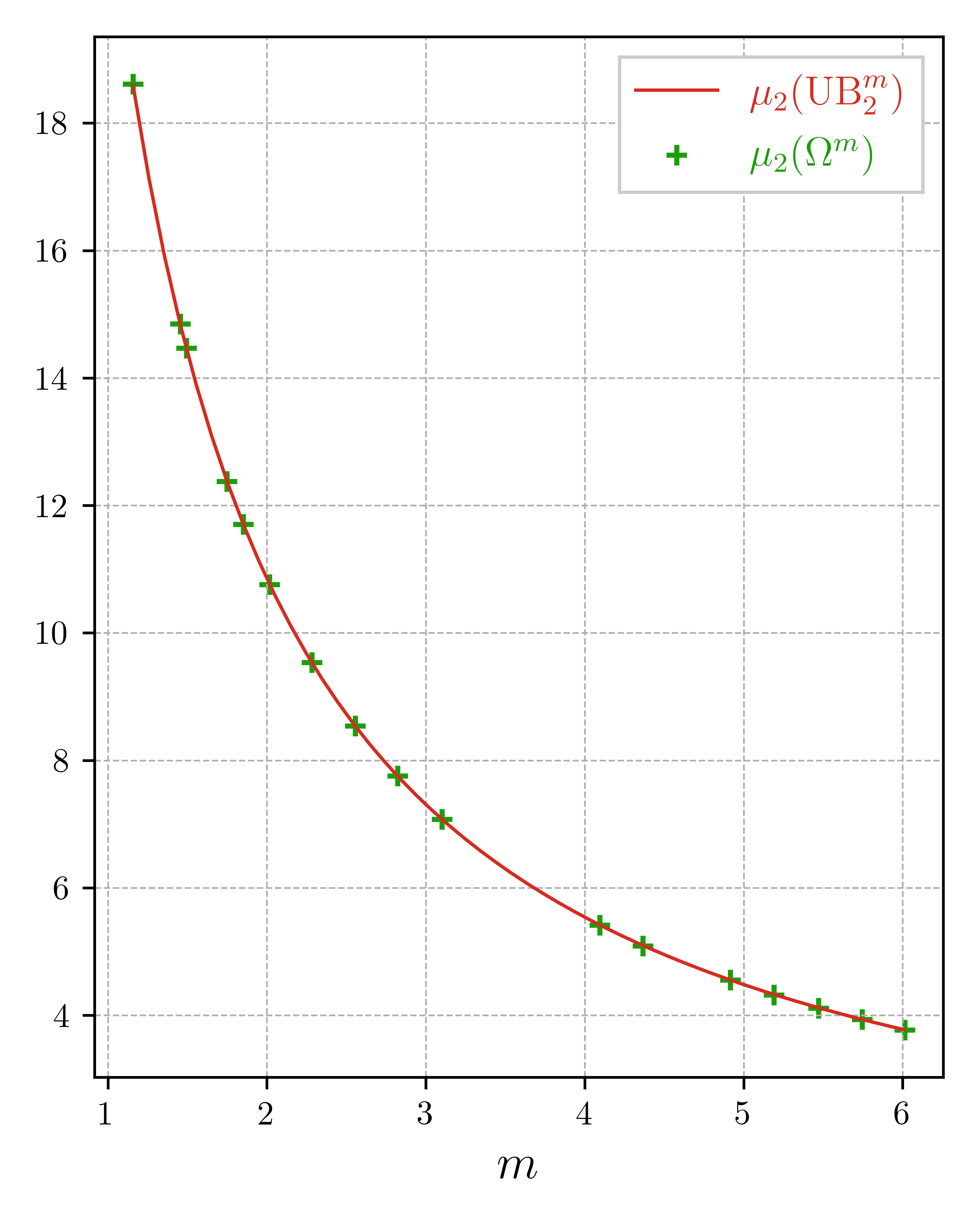}
    \includegraphics[width=0.49\textwidth]{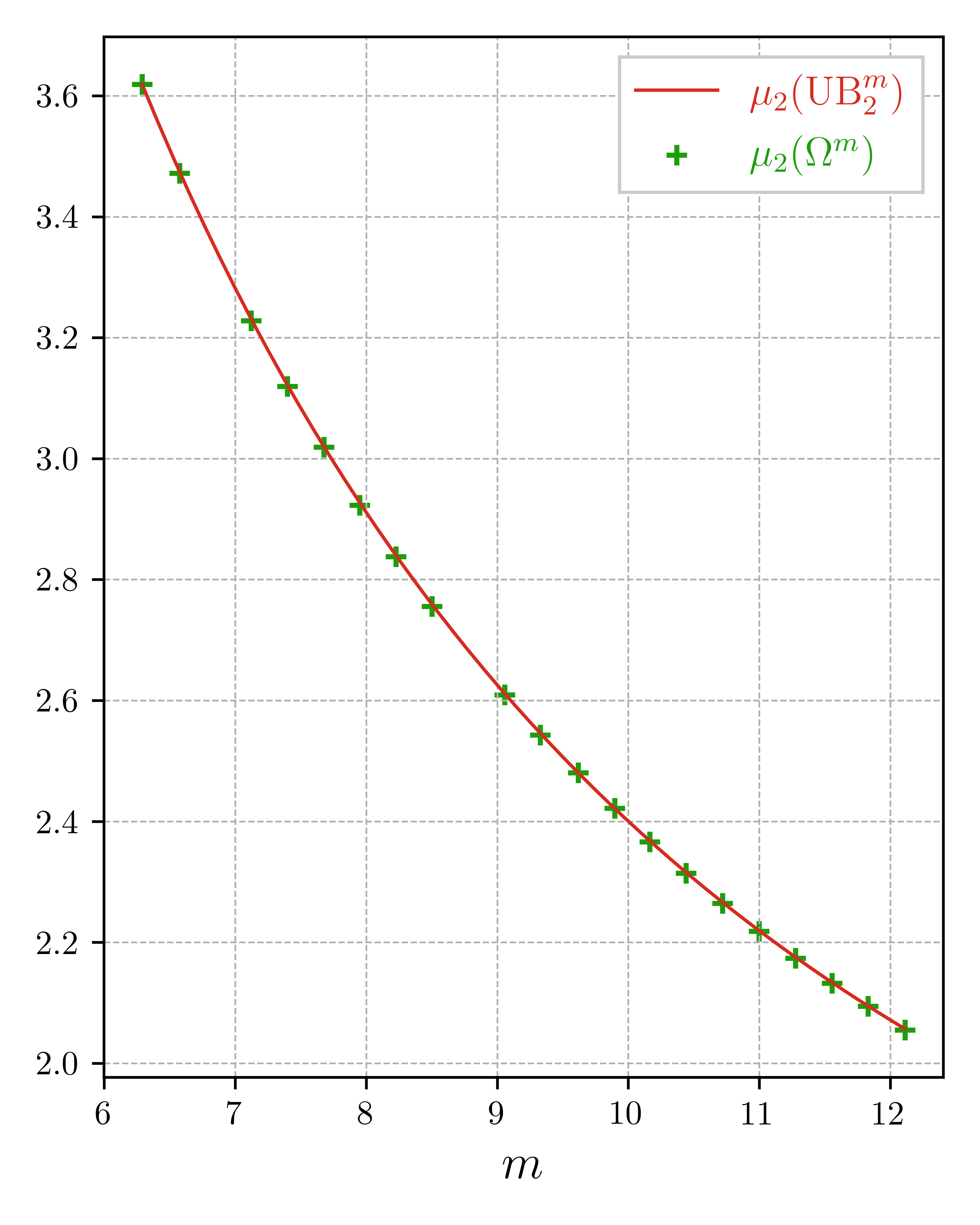}
    \caption{Optimal value of $\mu_2$ obtained by the level-set method.}
    \label{fig:mu_2_ls}
\end{figure}

The optimal results for $\mu_2$ are displayed Figure \ref{fig:mu_2_ls}. We can clearly see that the optimal shape is always the union of two spherical caps, as it has been proven in \cite{bucur_sharp_2022}. Hence this case can be considered as a test case to support the validity of the method. In Figure \ref{fig:mu_2_levelset_examples} are shown the optimal shapes. For $m=1.17$, we see that the computed shape isn't a union of two disjoint disks. Indeed, for large $m$, the first levelset procedure struggled to disconnect one domain into two disks due to numerical instabilities. The eigenvalue is however really close to the one of two disks.

\begin{figure}
    \centering
    \includegraphics[width=0.24\textwidth]{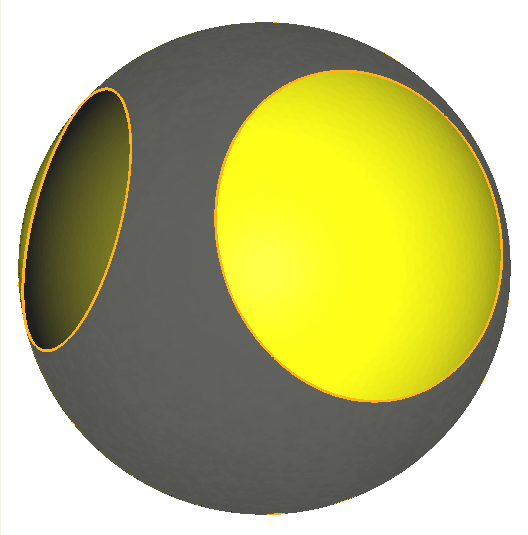}
    \includegraphics[width=0.24\textwidth]{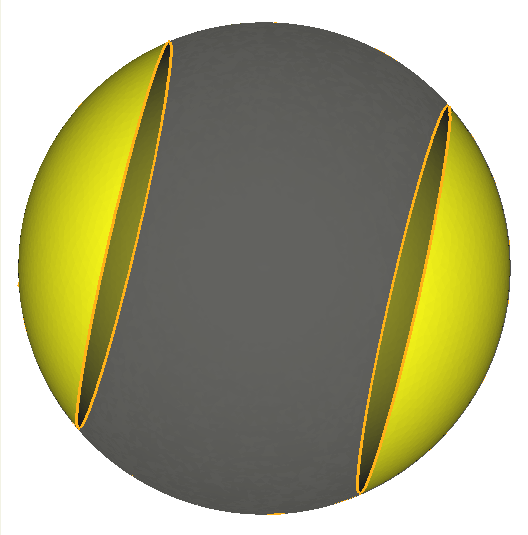}
    \includegraphics[width=0.24\textwidth]{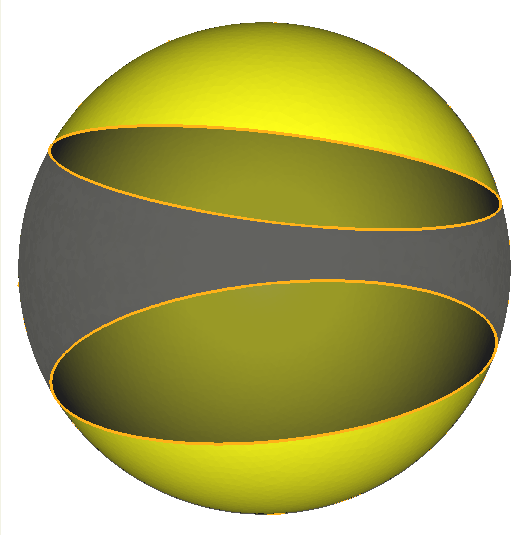}
    \includegraphics[width=0.24\textwidth]{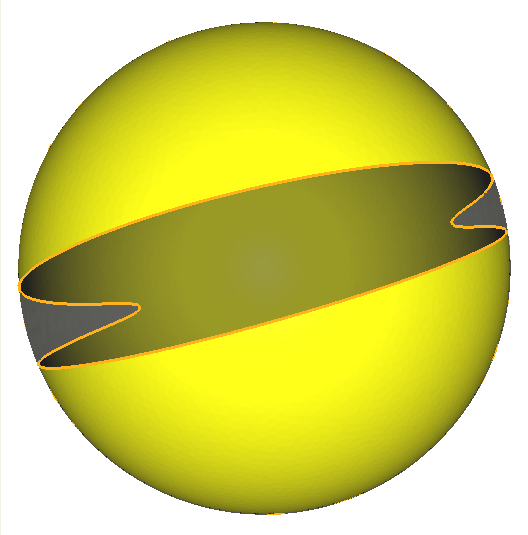}
    \caption{Example of optimal domains for $\mu_2$ for $m \in \{2.12, 5.1, 8.13, 11.17\}$.}
    \label{fig:mu_2_levelset_examples}
\end{figure}

\subsection{Optimization of $\mu_3$}

\begin{figure}
    \centering
    \includegraphics[width=0.49\textwidth]{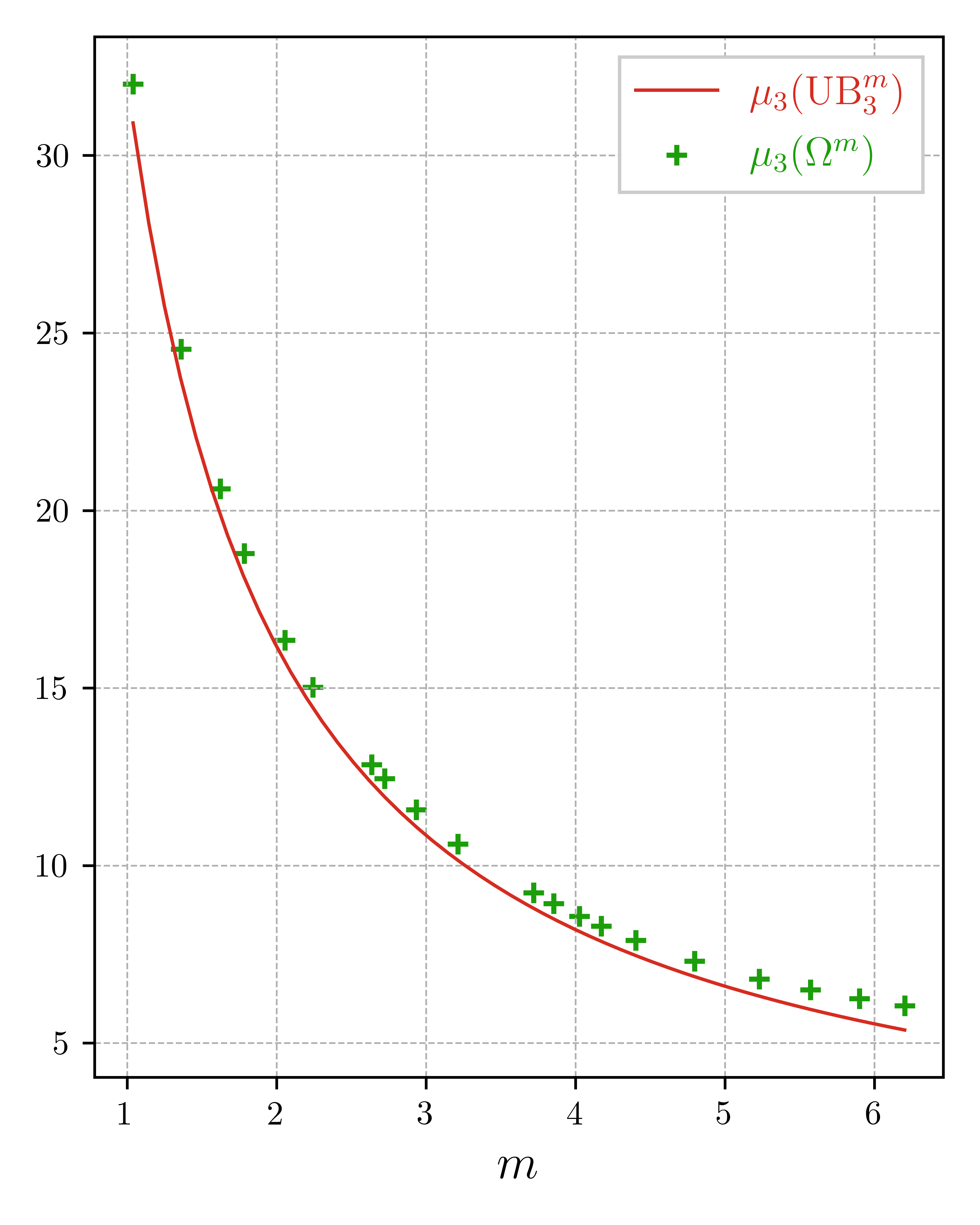}
    \includegraphics[width=0.49\textwidth]{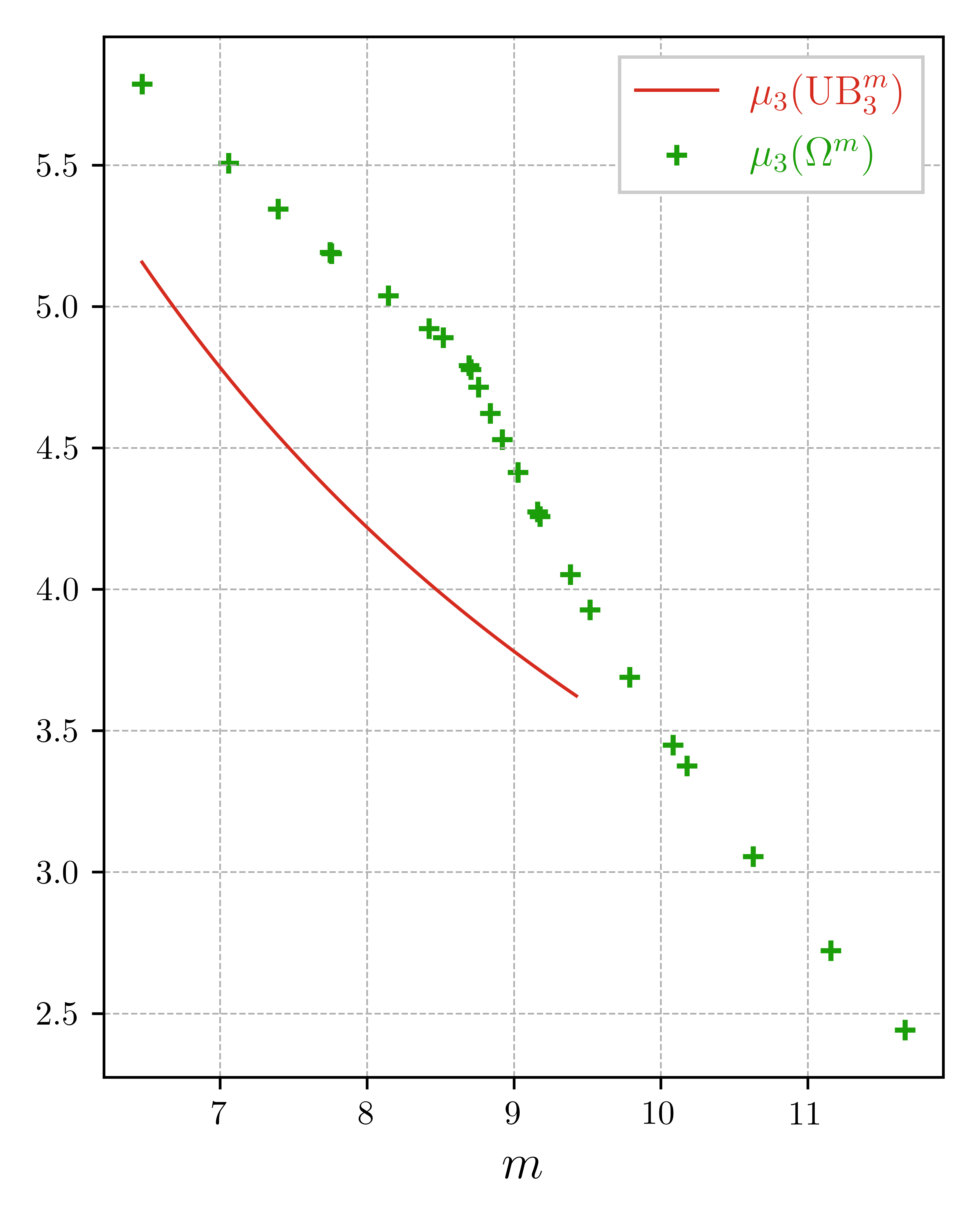}
    \caption{Optimal value of $\mu_3$ obtained by the level-set method.}
    \label{fig:mu_3_ls}
\end{figure}

In Figure \ref{fig:mu_3_ls} is displayed the results for the optimization of $\mu_3$. As for the density case, this eigenvalue shows a rich variety of behaviours depending on the value of $m$ (see Figure \ref{fig:mu_3_levelset_examples}).

\begin{figure}
    \centering
    \includegraphics[width=0.24\textwidth]{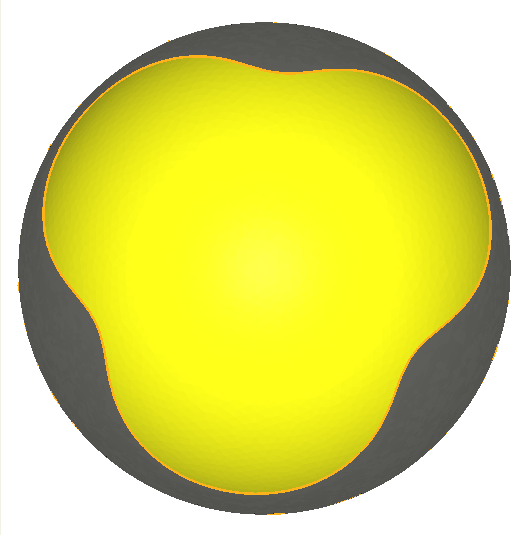}
    \includegraphics[width=0.24\textwidth]{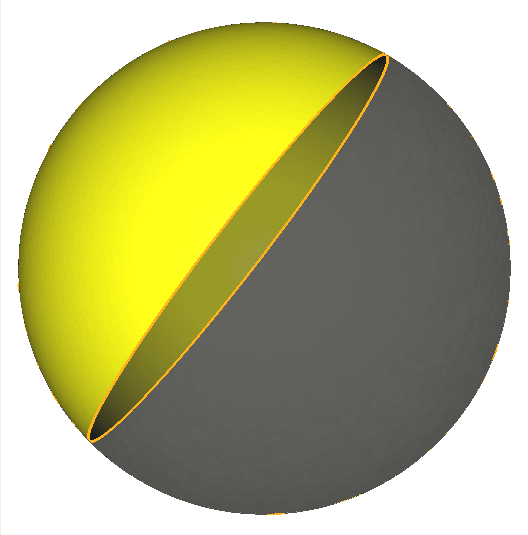}
    \includegraphics[width=0.24\textwidth]{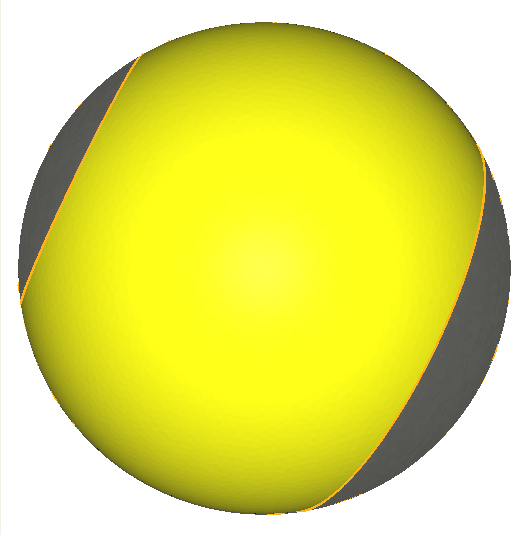}
    \includegraphics[width=0.24\textwidth]{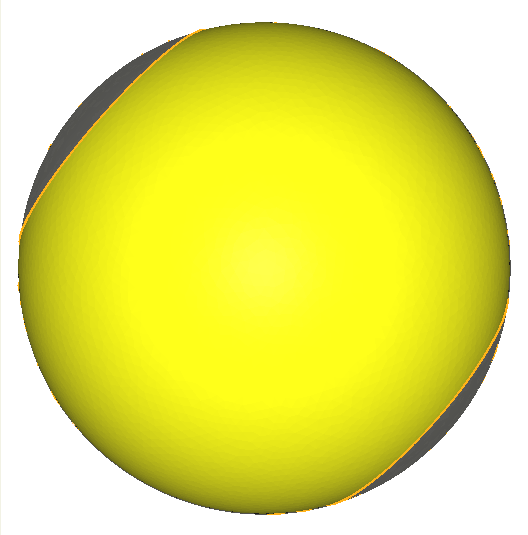}
    \caption{Example of optimal domains for $\mu_3$ for $m \in \{2.0, 5.22, 8.0, 11.04\}$.}
    \label{fig:mu_3_levelset_examples}
\end{figure}

The results seems to be in accordance with the ones given by the density method.

\subsection{Data.}

As for the results of the density method, all the final solutions are available in MEDIT format at \url{https://github.com/EloiMartinet/Neumann_Sphere/}, with a FreeFem++ script allowing to compute the eigenvalue and surface area of each solution.

\section{Explorations on a torus}

As it has been mentioned in the introduction, this last part is devoted to the optimization of eigenvalues on a torus. Specifically, we consider the torus $\mathbb{T}$ in $\R^3$ parametrized by
\[
  (u,v) \mapsto \left( (R+r\cos{v})\cos{u}, r \sin{v}, (R+r\cos{v})\sin{u} \right),
\]
that is, the torus with major radius $R$ and minor radius $r$. In the sequel, we will choose $R=2$ and $r=1$. Knowing that the surface area of $\mathbb{T}$ is given by
\[
  |\mathbb{T}| = 4 \pi^2 Rr
\]
we have in our case $|\mathbb{T}| \approx 78.96$.

The problem we consider now is analogous to the one on the sphere. With $\Om \subset \mathbb{T}$ being a Lipschitz open domain, the eigenvalue problem is
\[
  \begin{cases}
    -\Delta u = \mu_k(\Om) u \mbox { in } \Om,\\
    \frac{\partial u}{\partial n} = 0  \mbox { on } \partial \Om,
  \end{cases}.
\]

We also use the same notion of generalized eigenvalues of a density $\rho : \mathbb{T} \to [0,1]$, we define
\begin{equation*}
  \mu _k(\rho) := \inf_{V\in{\mathcal V}_{k+1}} \max_{u \in V \sm \{0\}} \frac{\int_{\mathbb{T}} \rho|\nabla u|^2}{\int_{\mathbb{T}} \rho u^2},
\end{equation*}
where ${\mathcal V}_{k+1}$ is the family of subspaces of dimension $k+1$ in
\[
  \{u\cdot 1_{\{\rho (x)>0\}}: u \in C^\infty_c (\mathbb{T})\}.
\]

We then consider the two optimization problems
\begin{equation*}
  \sup \left\{ \mu_k(\Om) \mbox{ s.t. } \Om \subset \mathbb{T}, | \Om | = m, \Om \mbox{ bounded, open and Lipschitz} \right\}
\end{equation*}

and

\begin{equation*}
  \sup \left\{ \mu_k(\rho) \mbox{ s.t. }  \rho : \mathbb{T} \rightarrow [0,1], \int_{\mathbb{T}}\rho =m\right\}.
\end{equation*}

with $0 < m < |\mathbb{T}|$ and $k>0$.

This last formulation allows to perform the same density method performed on the sphere. These are the results presented hereafter, followed by the results obtained by the level set method.

\subsection{Density optimization}

We precise that the optimization parameters that have been used are the same than the ones used for the sphere. We begin to depict some of the optimal densities in Figure \ref{fig:tore_mu_density_examples}. For a better visualization, we recall that all the meshes and densities are available at \url{https://github.com/EloiMartinet/Neumann_Sphere/}.

\begin{figure}
    \centering
    \includegraphics[width=0.24\textwidth]{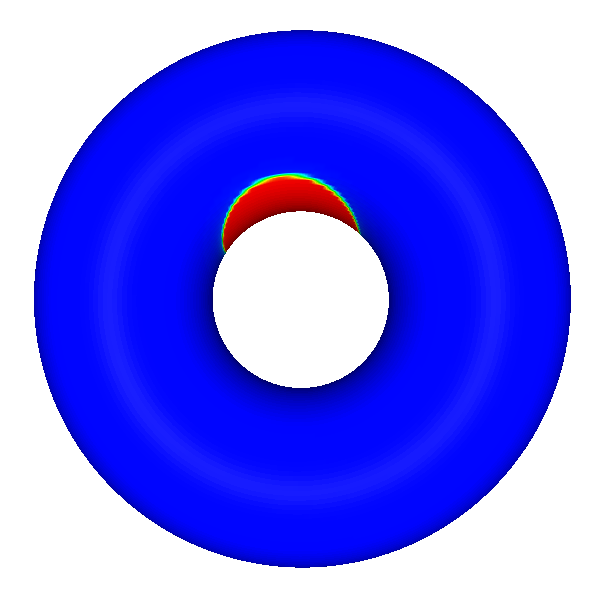}
    \includegraphics[width=0.24\textwidth]{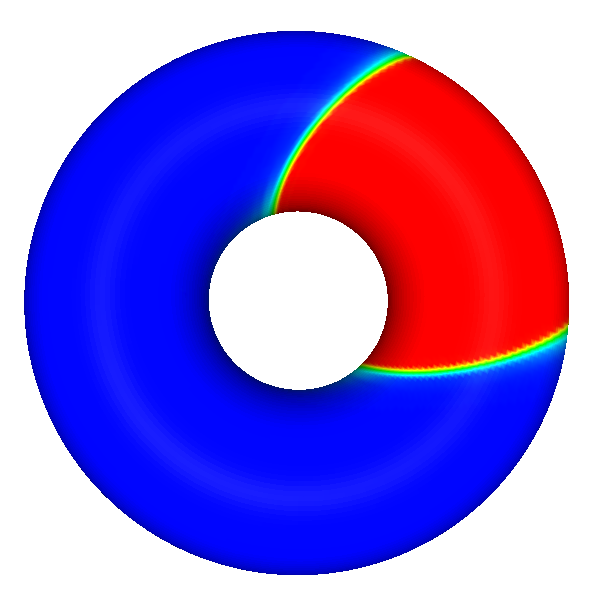}
    \includegraphics[width=0.24\textwidth]{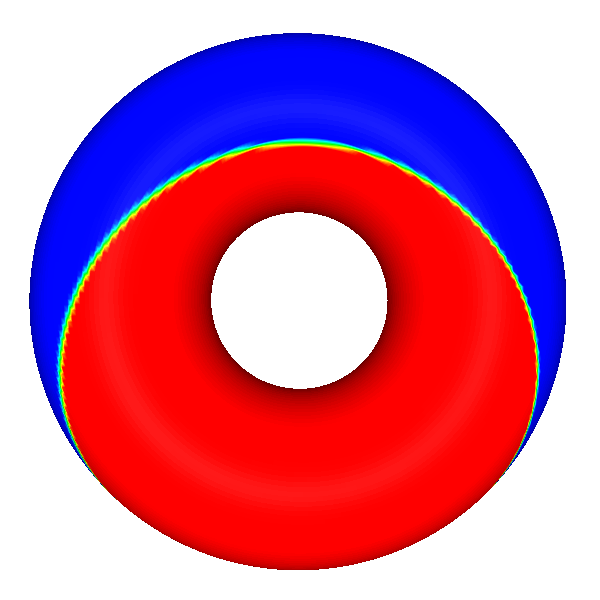}
    \includegraphics[width=0.24\textwidth]{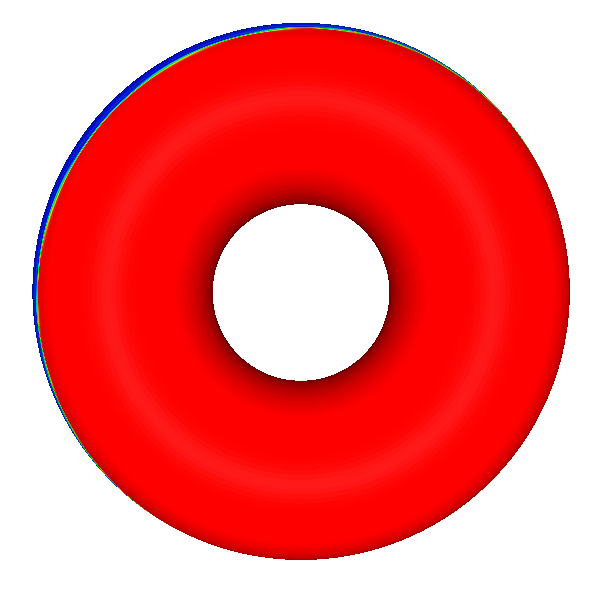}

    \includegraphics[width=0.24\textwidth]{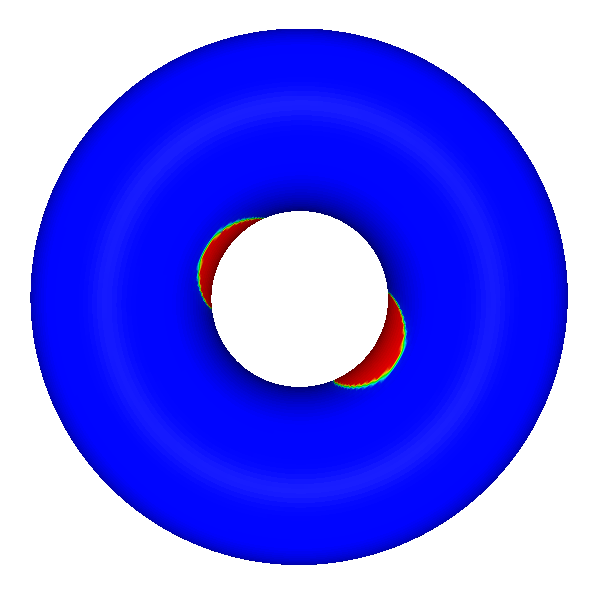}
    \includegraphics[width=0.24\textwidth]{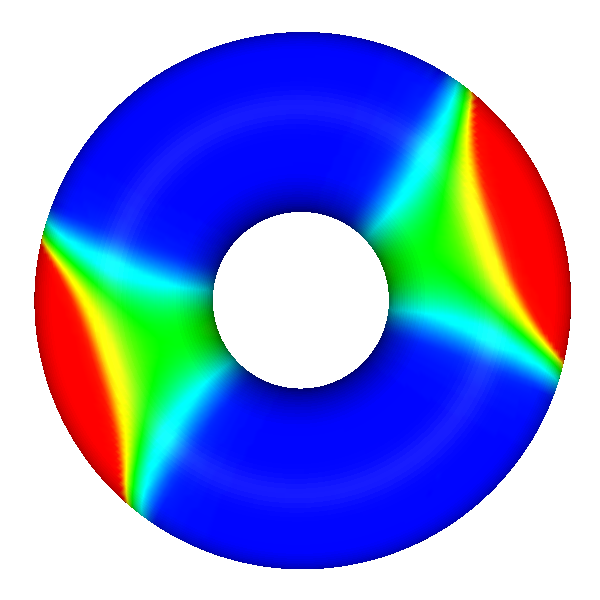}
    \includegraphics[width=0.24\textwidth]{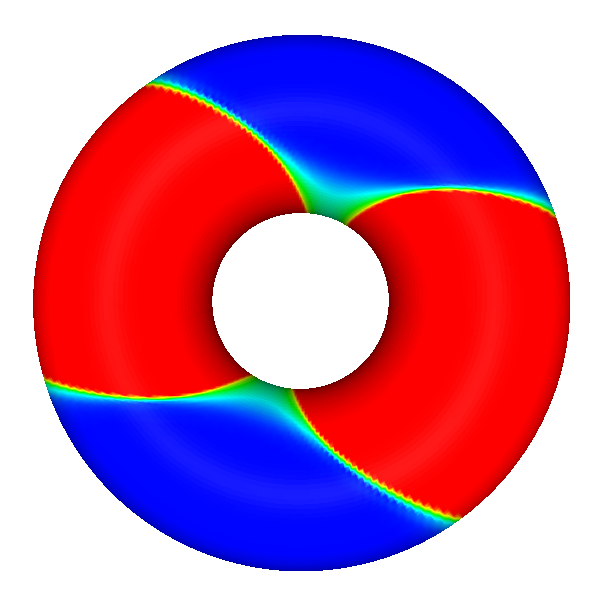}
    \includegraphics[width=0.24\textwidth]{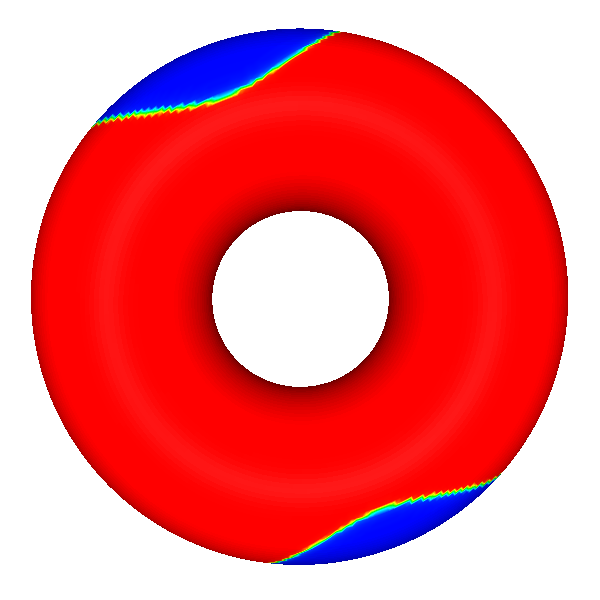}

    \includegraphics[width=0.24\textwidth]{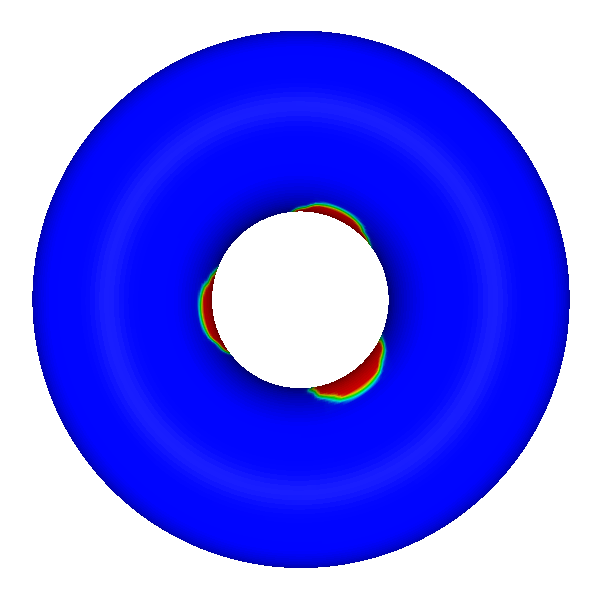}
    \includegraphics[width=0.24\textwidth]{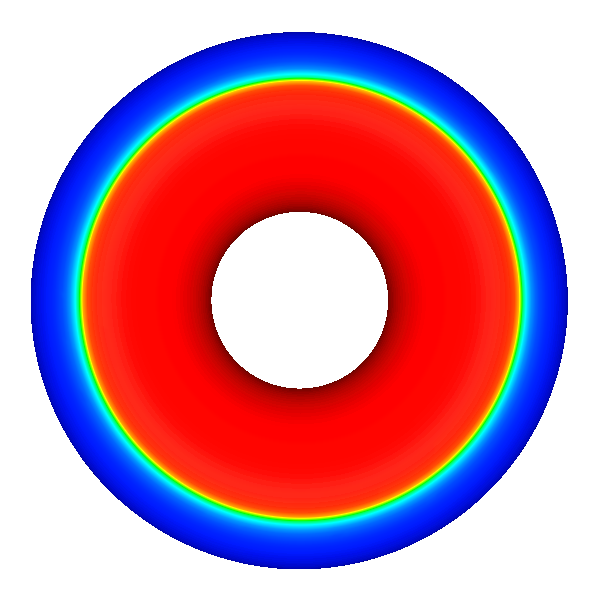}
    \includegraphics[width=0.24\textwidth]{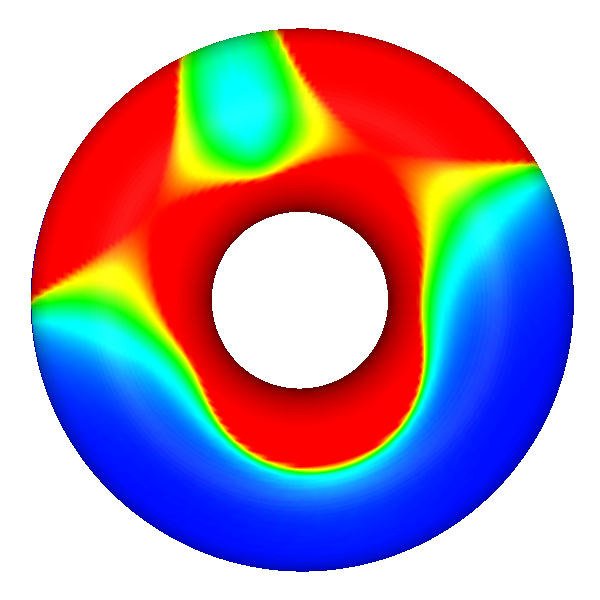}
    \includegraphics[width=0.24\textwidth]{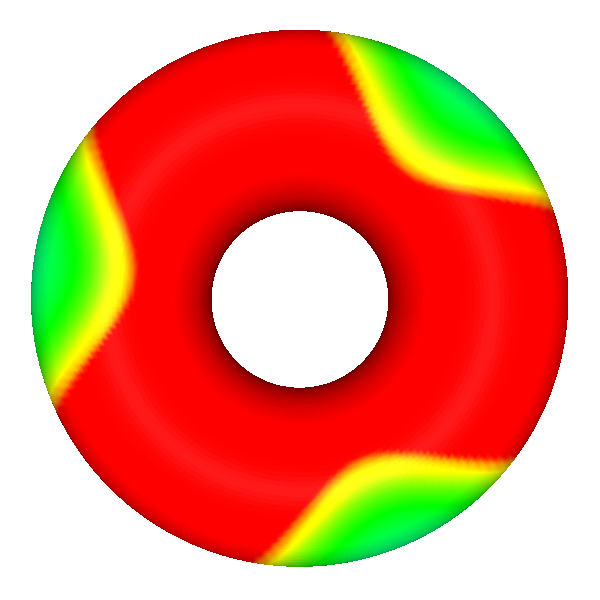}

    \caption{Example of optimal densities for $\mu_1, \mu_2$ and $\mu_3$ (resp. first, second and third row) for $m \in \{3, 22, 43, 68\}$ approximately.}
    \label{fig:tore_mu_density_examples}
\end{figure}

We can notice that for small enough masses, the optimal densities seems to be the characteristic function of geodesic balls for $k\in\{1,2,3\}$. However, by comparison with the case of the plane and the sphere, the case of $\mu_3$ must be taken with caution.
A striking fact is that contrary to the case of the sphere, the optimal density for $\mu_1$ seems to stay a charateristic function whereas we cans
witness some homogenization for $\mu_2$.
The optimal eigenvalues plotted as functions of $m$ are shown in Figure \ref{fig:tore_mu_density}.

\begin{figure}
    \centering
    \includegraphics[width=0.4\textwidth]{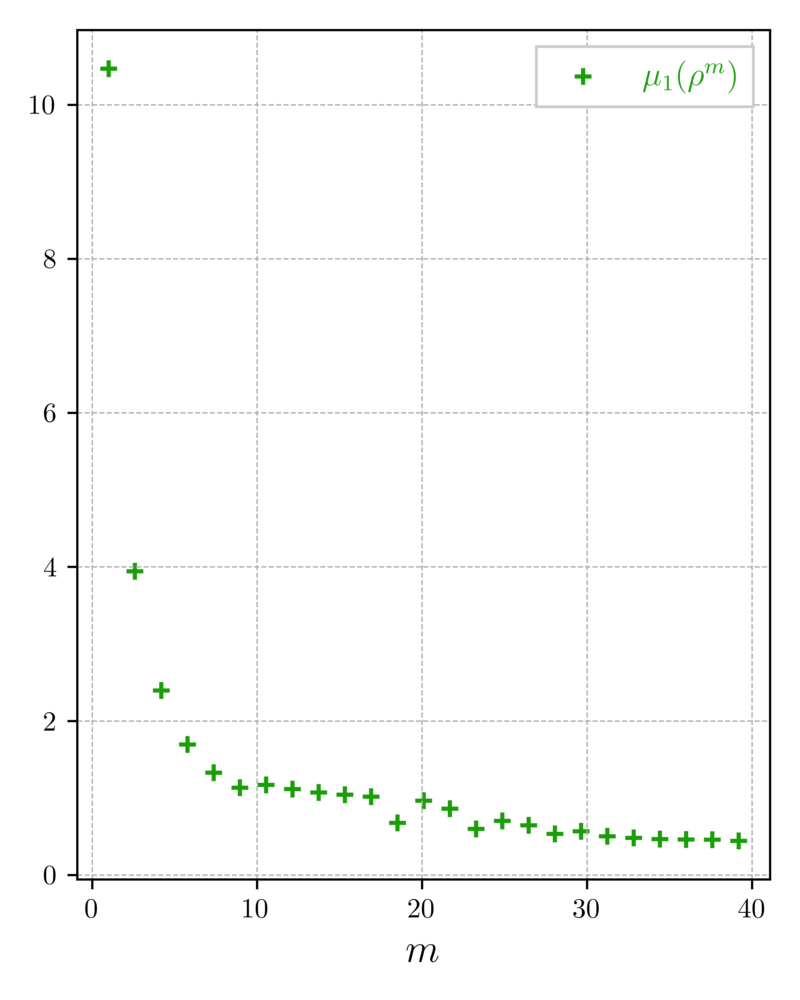}
    \includegraphics[width=0.4\textwidth]{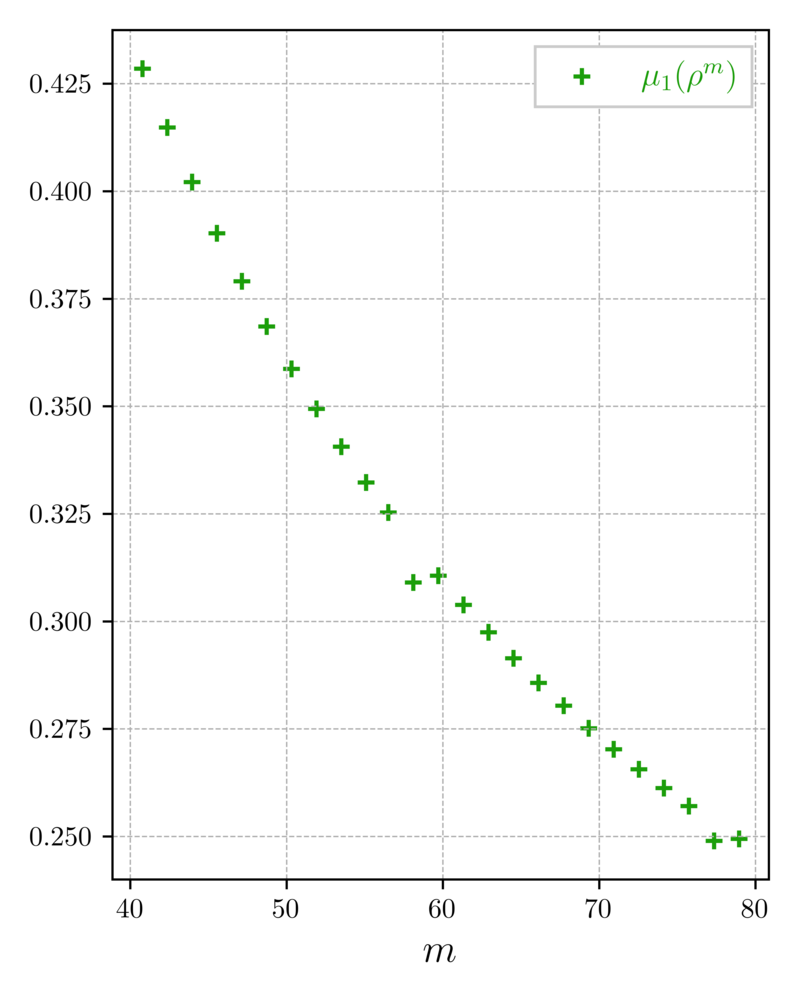}
    \includegraphics[width=0.4\textwidth]{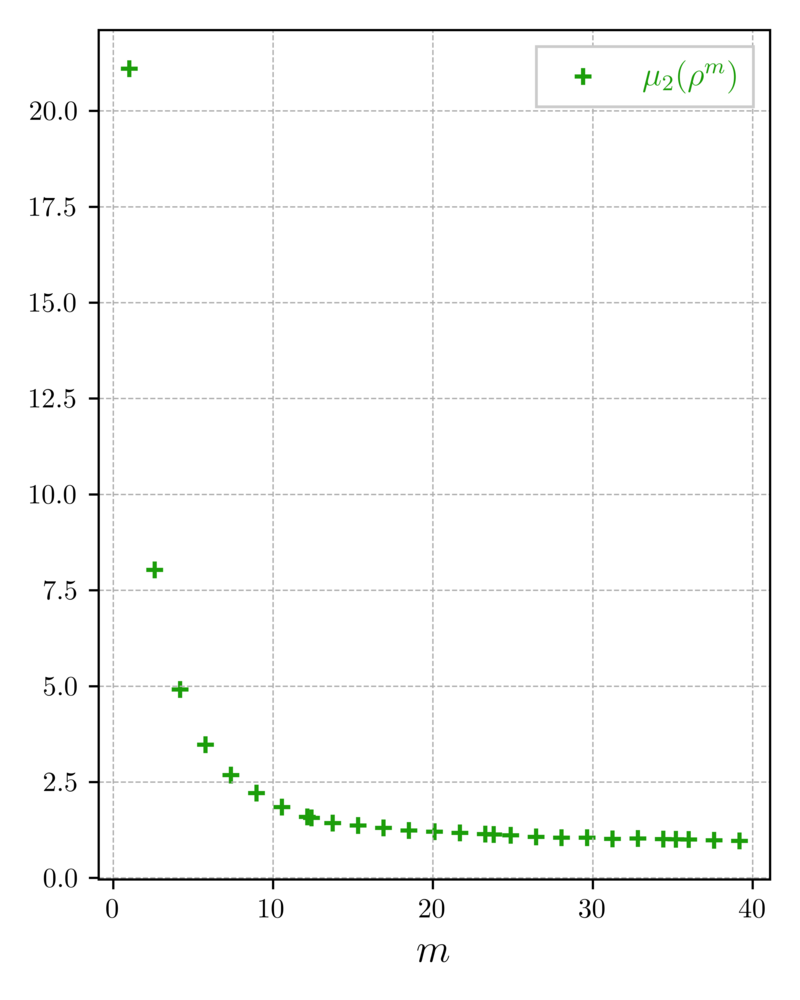}
    \includegraphics[width=0.4\textwidth]{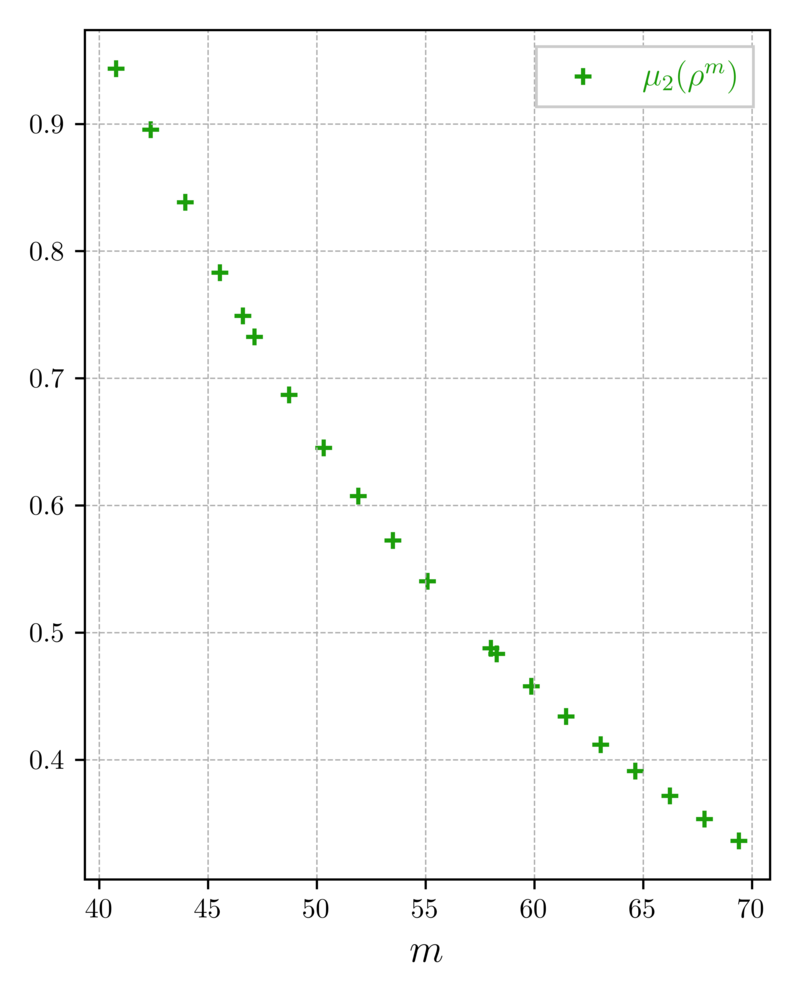}
    \includegraphics[width=0.4\textwidth]{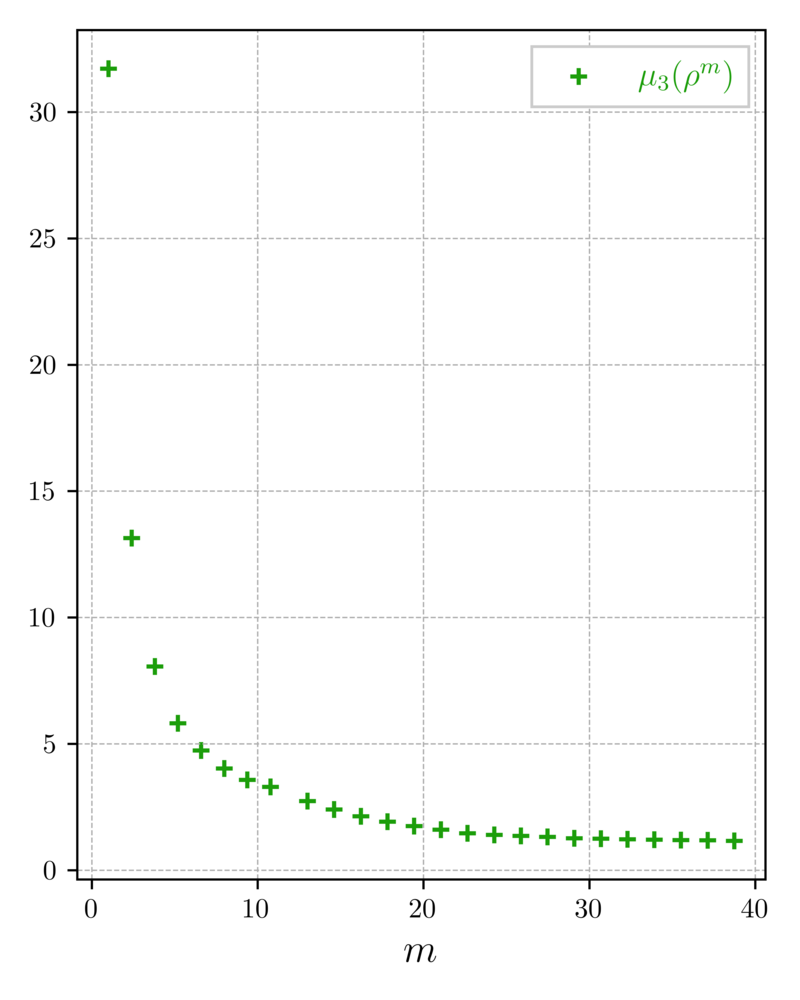}
    \includegraphics[width=0.4\textwidth]{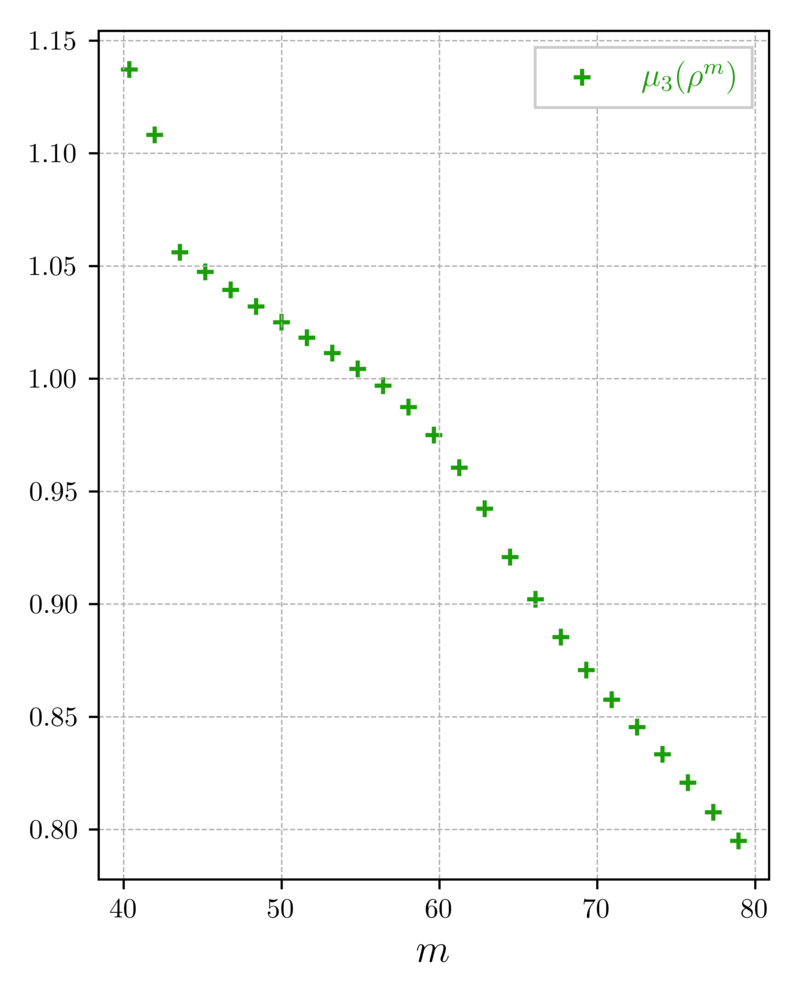}
    \caption{Optimal values of $\mu_1, \mu_2$ and $\mu_3$ (resp. first, second and third row) obtained by the density method.}
    \label{fig:tore_mu_density}
\end{figure}

\subsection{Level set optimization}

In Figure \ref{fig:tore_mu_ls_examples} are the optimal domains obtained by the level set method for $\mu_1, \mu_2, \mu_3$ and various masses.

\begin{figure}
    \centering
    \includegraphics[width=0.24\textwidth]{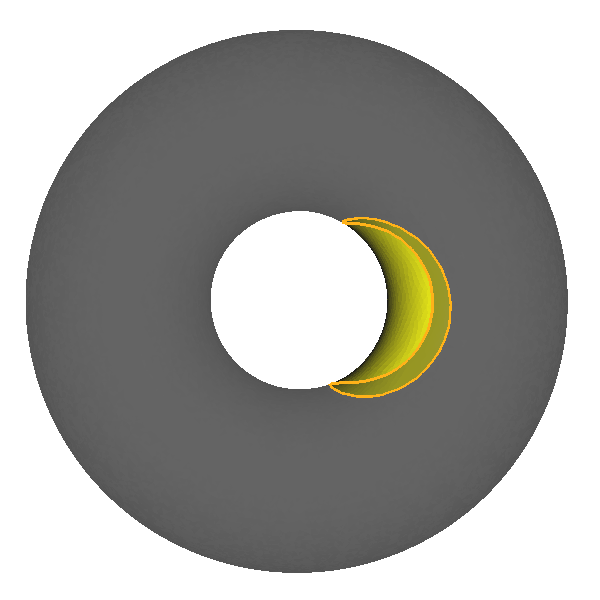}
    \includegraphics[width=0.24\textwidth]{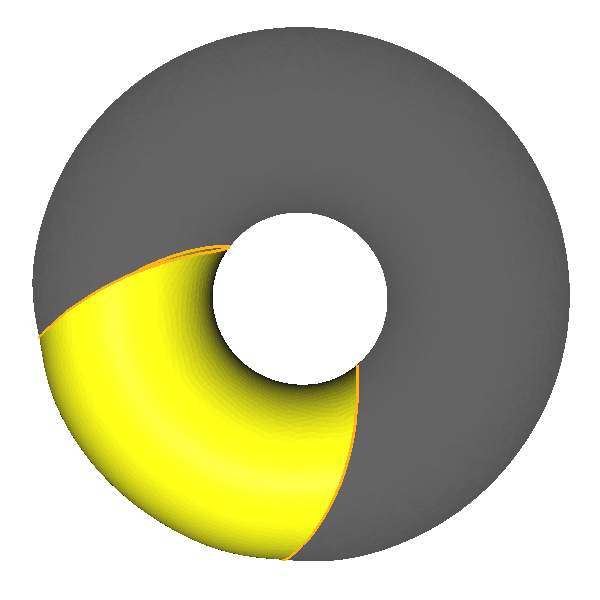}
    \includegraphics[width=0.24\textwidth]{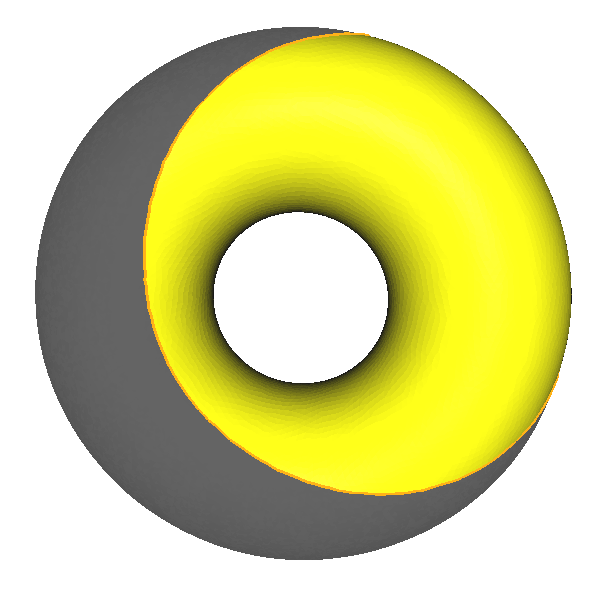}
    \includegraphics[width=0.24\textwidth]{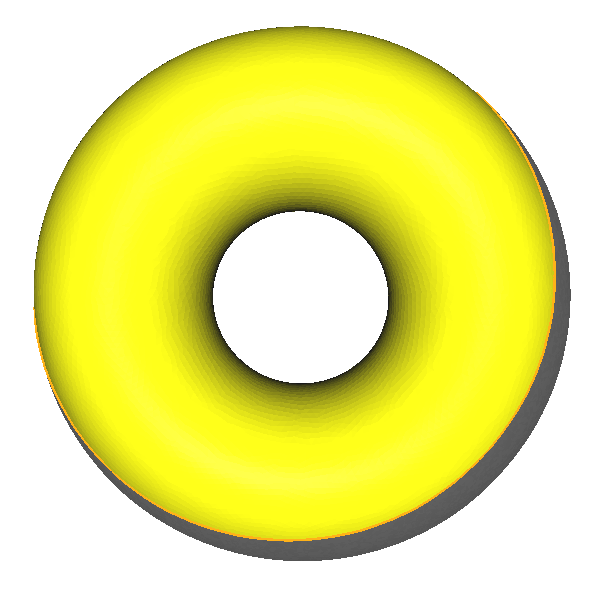}

    \includegraphics[width=0.24\textwidth]{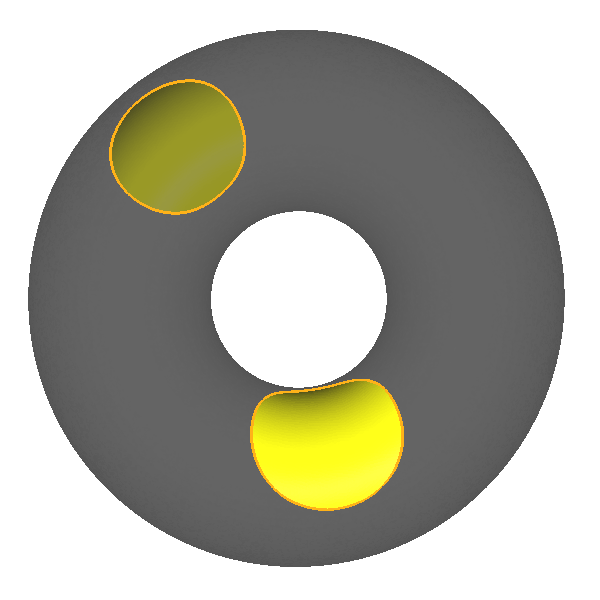}
    \includegraphics[width=0.24\textwidth]{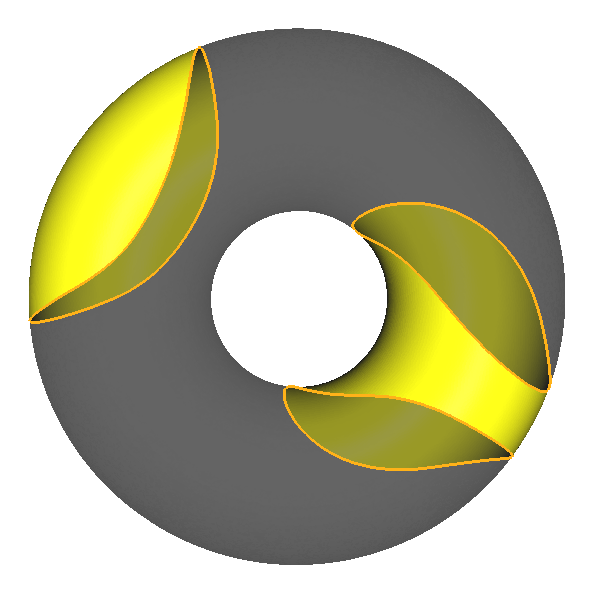}
    \includegraphics[width=0.24\textwidth]{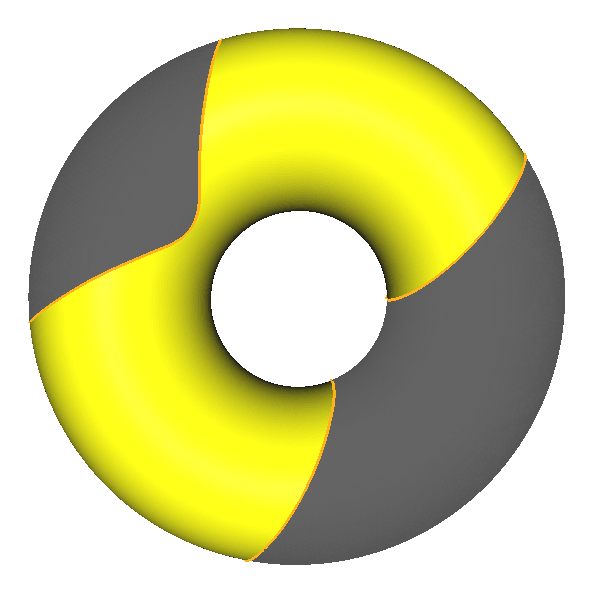}
    \includegraphics[width=0.24\textwidth]{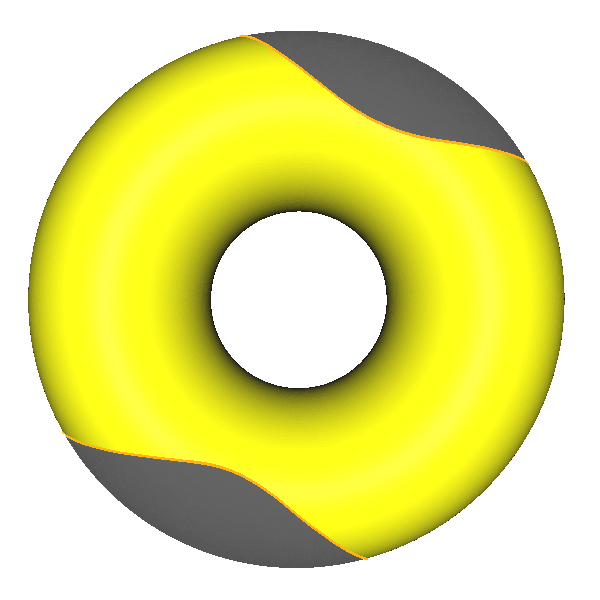}

    \includegraphics[width=0.24\textwidth]{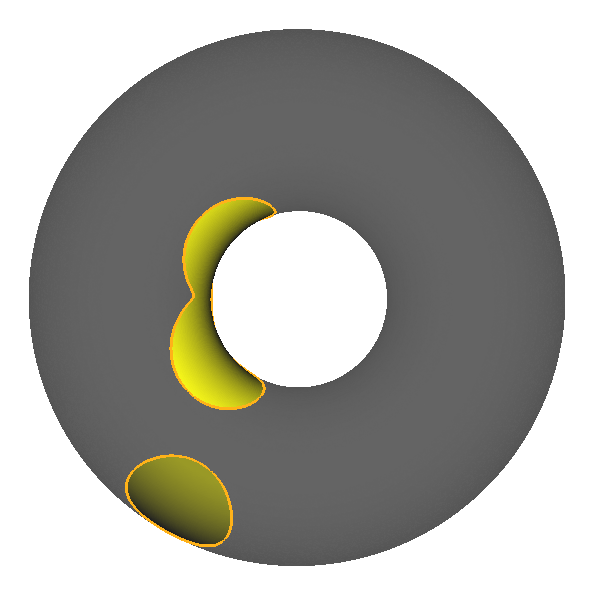}
    \includegraphics[width=0.24\textwidth]{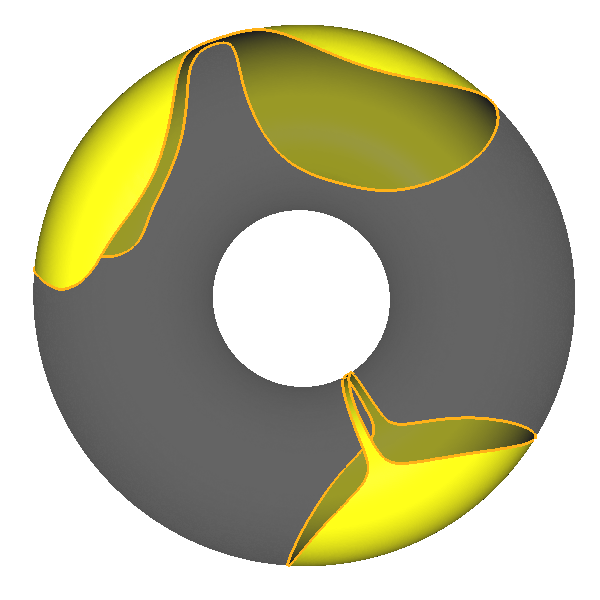}
    \includegraphics[width=0.24\textwidth]{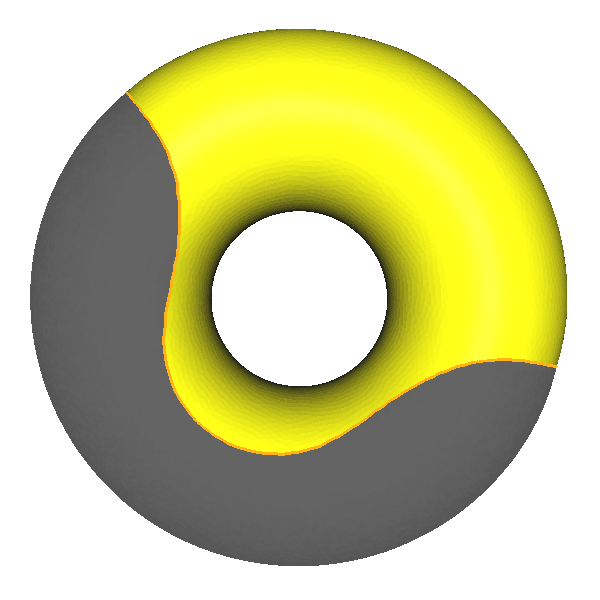}
    \includegraphics[width=0.24\textwidth]{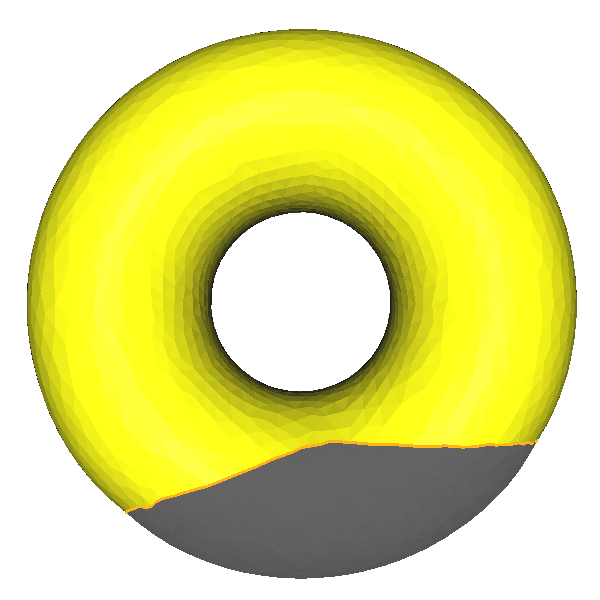}

    \caption{Example of optimal domains for $\mu_1, \mu_2$ and $\mu_3$ (resp. first, second and third row) for $m \in \{4, 22, 42, 60\}$ approximately.}
    \label{fig:tore_mu_ls_examples}
\end{figure}

The optimal eigenvalues plotted as functions of $m$ are shown in Figure \ref{fig:tore_mu_ls}.

\begin{figure}
    \centering
    \includegraphics[width=0.4\textwidth]{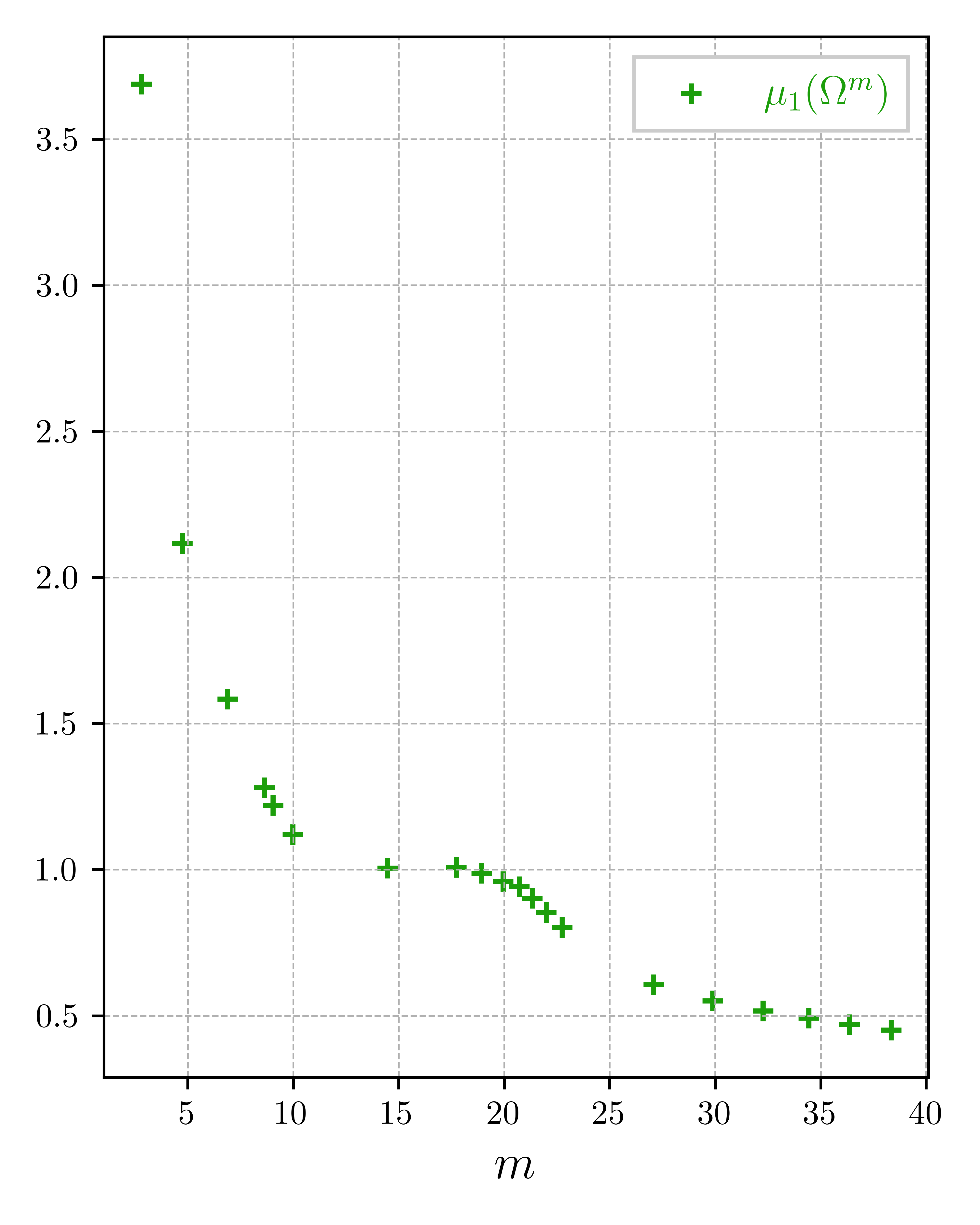}
    \includegraphics[width=0.4\textwidth]{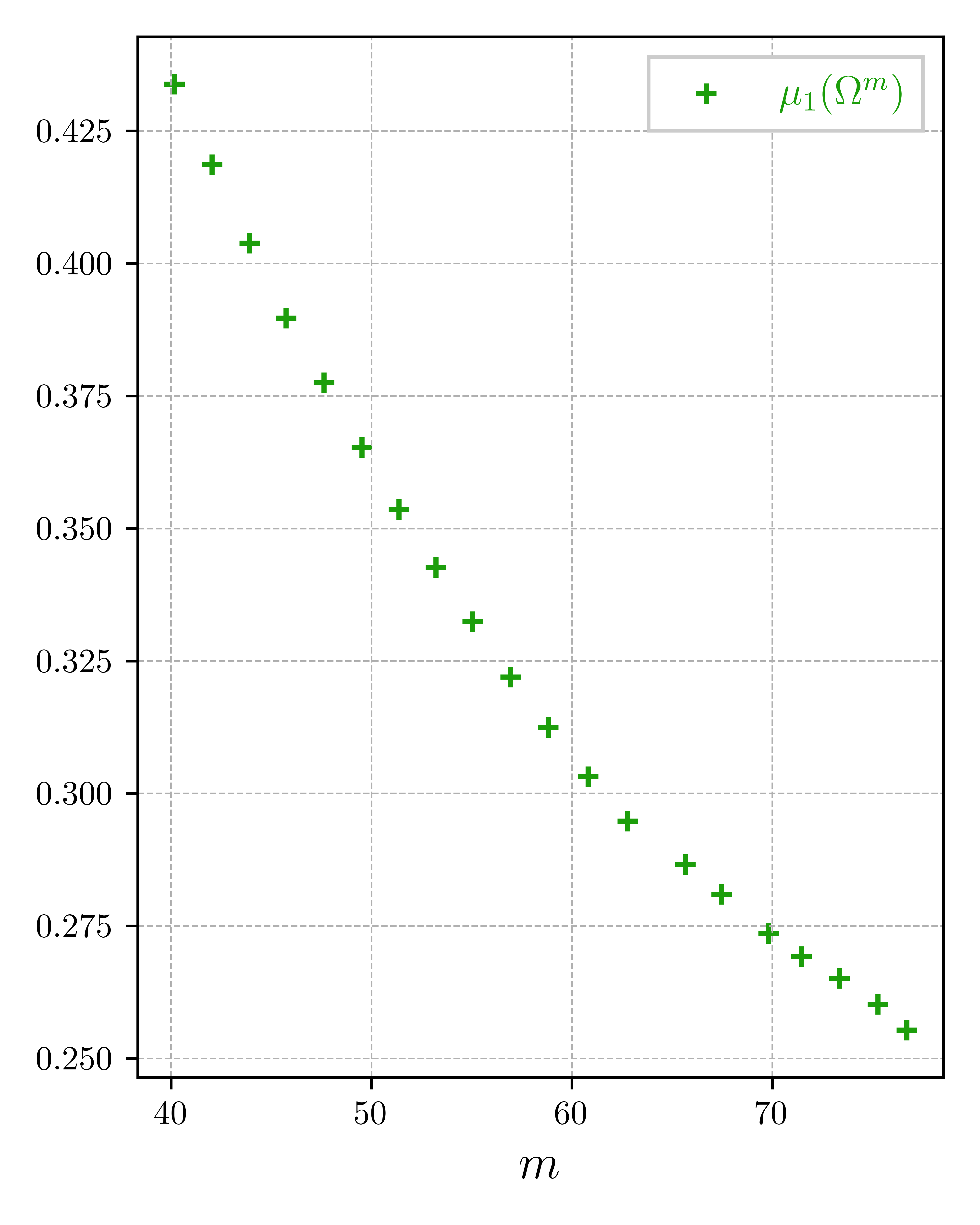}
    \includegraphics[width=0.4\textwidth]{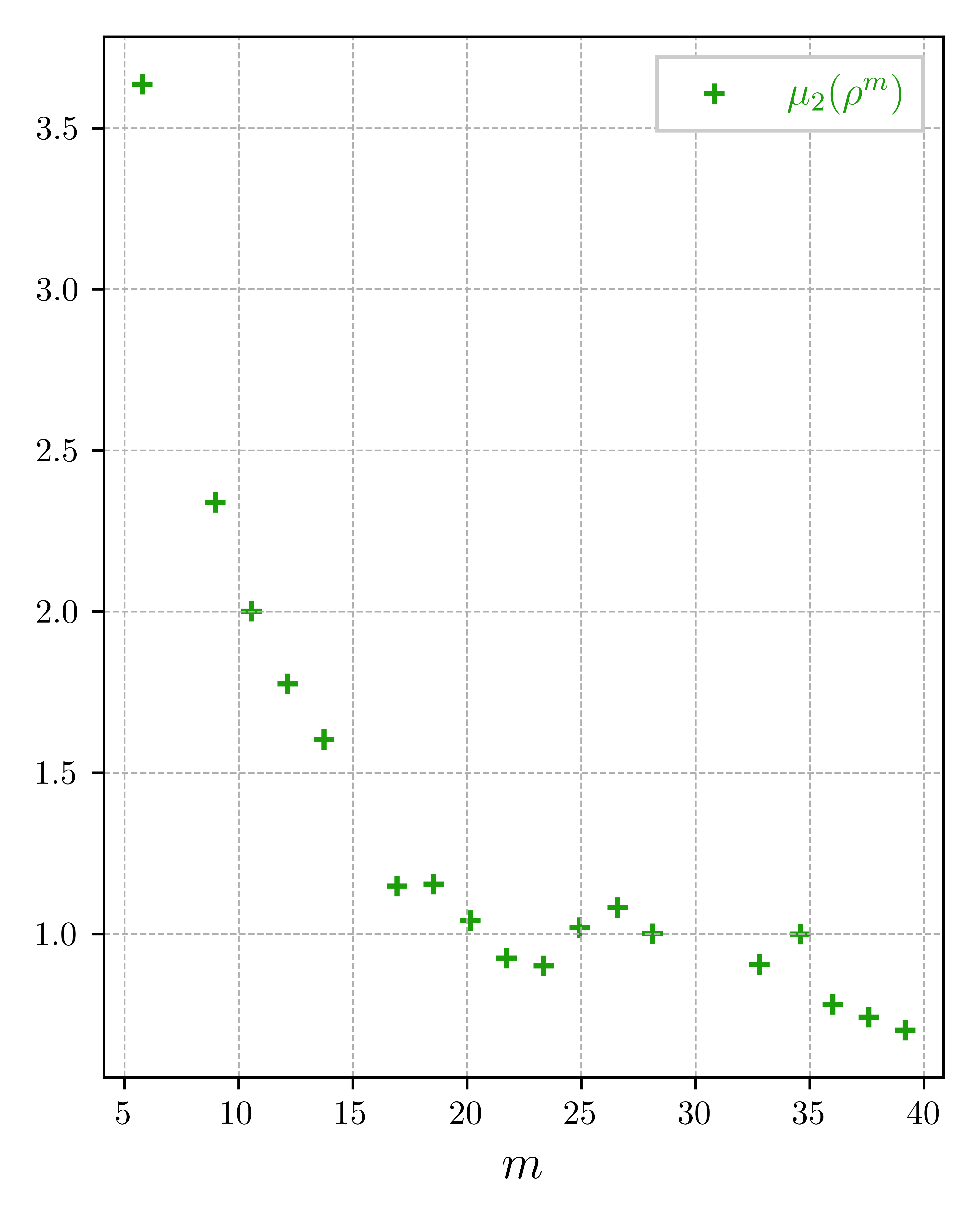}
    \includegraphics[width=0.4\textwidth]{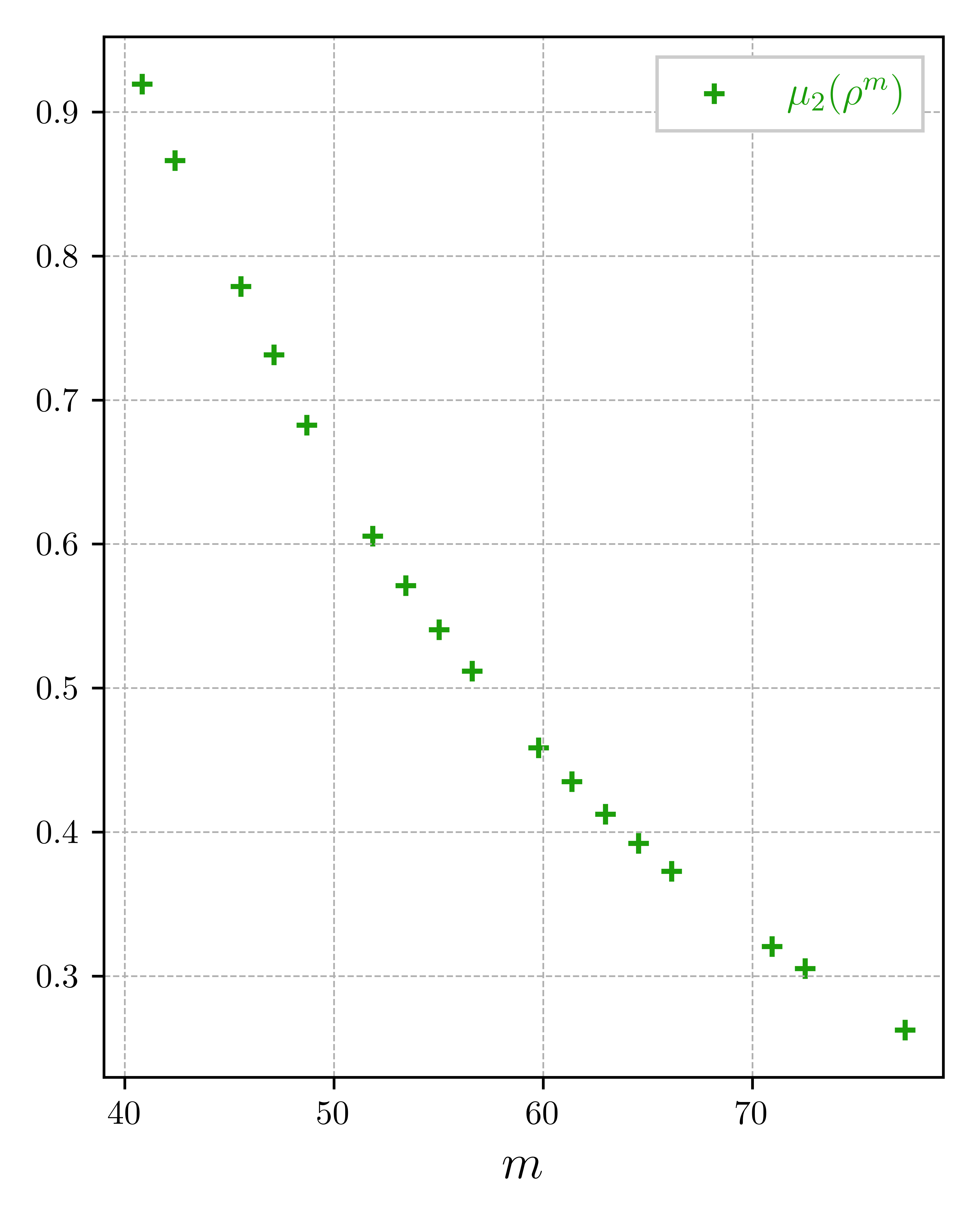}
    \includegraphics[width=0.4\textwidth]{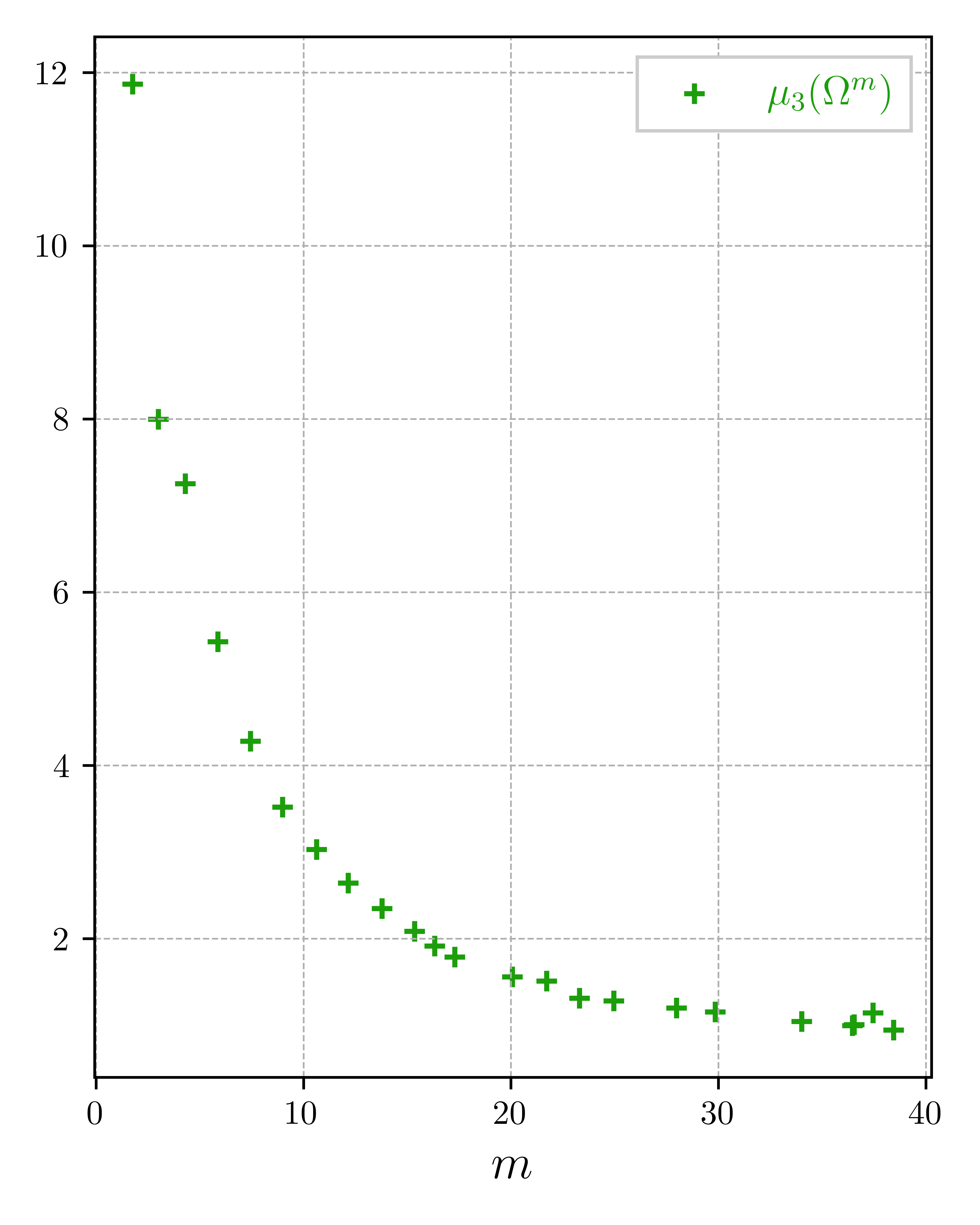}
    \includegraphics[width=0.4\textwidth]{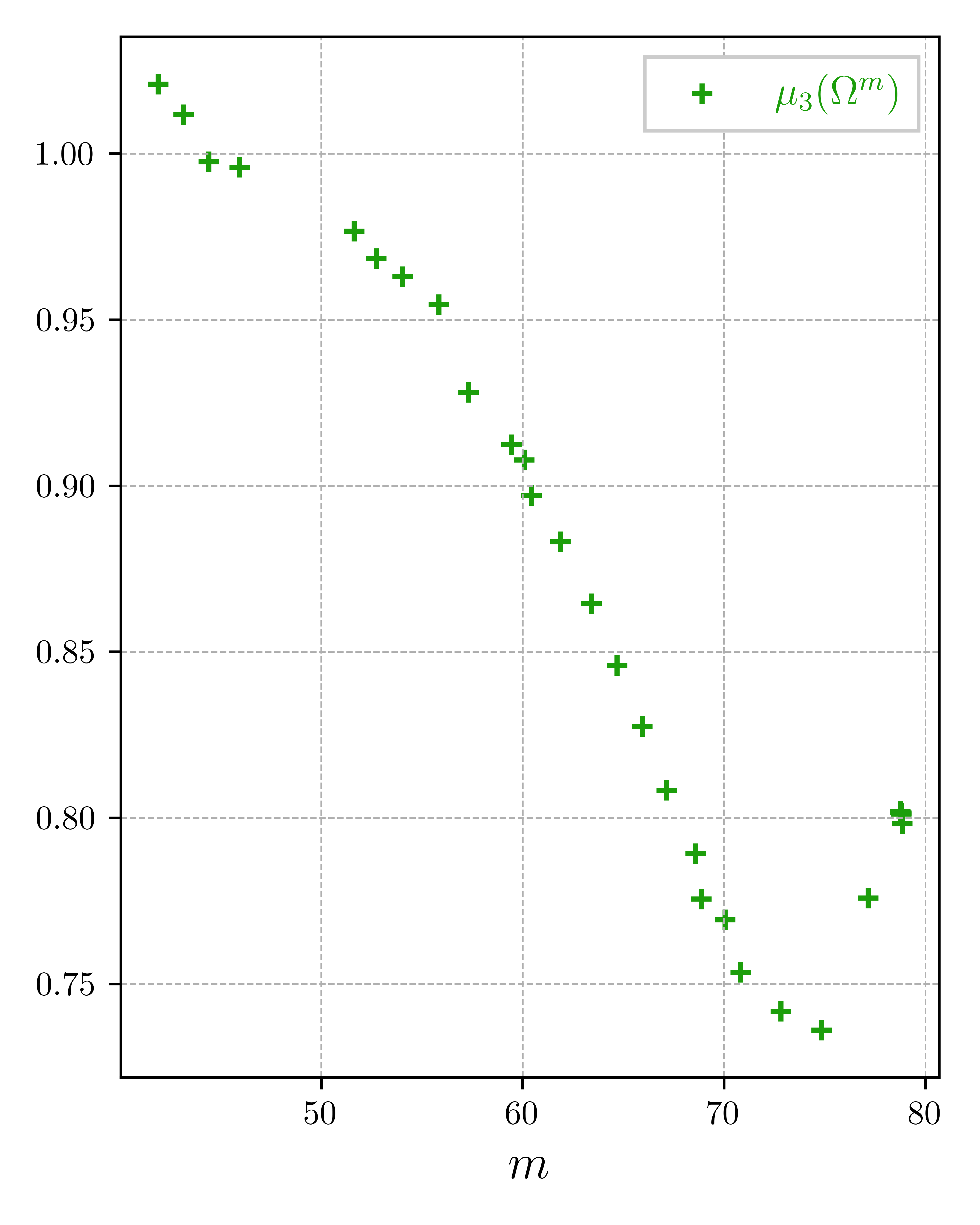}
    \caption{Optimal values of $\mu_1, \mu_2$ and $\mu_3$ (resp. first, second and third row) obtained by the level set method.}
    \label{fig:tore_mu_ls}
\end{figure}

\section{Discussion}

We have presented two ways to explore numerically the optimization of eigenvalues of the Laplace-Beltrami operator
with Neumann boudndary conditions for domains in the sphere and the torus. The first way generalizes the notion
of eigenvalues of domains and thus is independant of topological consideration. On the other hand, the second
method relies on the level set representation of the domain, which allows topological changes.

In the case of the optimization on the sphere,  this flexibility turned out to be an important feature, regarding the topological complexity of certain optimal domains for $\mu_1$ on the sphere. Indeed, while the density optimization leads to optima that does not corresponds to domains, the level set procedure tries to create areas with a lot of holes, which might indicate non-existence of optimal domains for large enough surface area. Oppositely, it seems clear that for small enough surface area, the optimal density (and thus, optimal domain) for $\mu_1$ on the sphere is the one of a geodesic ball.
It has then been witnessed that the behaviour of $\mu_2$ on the sphere was completely clear. Indeed, the optimal domains are always the union of two geodesic balls as it have been shown in \cite{bucur_sharp_2022}.
For $\mu_3$ , it might be difficult to decribe theoretically the way optima acts depending on the total surface area but it would be interesting to prove some necessary conditions such domains have to meet, such as symmetry.

In the case of the optimization on the torus, we noticed that while homogenization happened this time for $\mu2$ and not for $\mu_1$. Further investigations may be needed to understand better how this phenomenon evolves with the size of the torus, and if it such homogenization happens in flat tori like $(\R^n/\Z^n)$.

In any case, an interesting result would be to get a better grasp on the existence or non-existence of optimal domains on manifold.

\section{Acknowledgement}

The author was supported by the ANR SHAPO (ANR-18-CE40-0013).

\clearpage
\bibliographystyle{plain}
\bibliography{references}

\end{document}